\documentclass{amsart}

\usepackage[english]{babel}
\usepackage{amsmath,amssymb,amsthm}
\usepackage[all]{xy}
\usepackage{enumerate}
\usepackage{graphicx}
\usepackage{vmargin}
\usepackage{hyperref}

\newtheorem{thm}{Theorem}[section]
\newtheorem{prop}[thm]{Proposition}
\newtheorem{lem}[thm]{Lemma}
\newtheorem{cor}[thm]{Corollary}
\newtheorem{qst}[thm]{Question}
\theoremstyle{definition}

\theoremstyle{remark}
\newtheorem{rmq}[thm]{Remark}

\newtheorem{exe}[thm]{Example}

\numberwithin{equation}{section}

\newcommand\BB{\mathbb{B}}
\newcommand\R{\mathbb{R}}
\newcommand\C{\mathbb{C}}
\newcommand\F{\mathbb{F}}
\newcommand\Z{\mathbb{Z}}
\newcommand\N{\mathbb{N}}

\newcommand\TT{\mathbb{T}}
\newcommand\Ss{\mathbb{S}}

\newcommand\Cinf{\mathbb{C}_{\infty}}
\newcommand\Kinf{K_{\infty}}
\newcommand\vinf{v_{\infty}}
\newcommand{\per}{\tilde \pi}
\newcommand\lin{\operatorname{lin}}

\newcommand\ii{\mathbf{i}}
\newcommand\XX{\mathbf{X}}
\newcommand\ttt{\mathbf{t}}

\newcommand\PP{\mathfrak{P}}

\newcommand\rdeg{\operatorname{rdeg}}
\newcommand\mdeg{\mathbf{deg}}

\newcommand\fdot{\raisebox{-0.25ex}{\scalebox{1.5}{$\cdot$}}}

\begin{document}

\title{Anderson-Stark Units for ${\mathbb F}_{q}[\theta]$}
\date{\today}
\author[B. Angl\`es \and F. Pellarin \and F. Tavares Ribeiro]{Bruno Angl\`es\textsuperscript{1} \and Federico Pellarin\textsuperscript{2,3}   \and Floric Tavares Ribeiro\textsuperscript{1}}

\address{
Normandie Universit\'e,
Universit\'e de Caen Basse-Normandie, 
CNRS UMR 6139, 
Campus II, Boulevard Mar\'echal Juin, 
B.P. 5186, 
14032 Caen Cedex, France.
}
\email{bruno.angles@unicaen.fr, floric.tavares-ribeiro@unicaen.fr}
\address{ 
Institut Camille Jordan, UMR 5208
Site de Saint-Etienne,
23 rue du Dr. P. Michelon,
42023 Saint-Etienne, France
}
\email{federico.pellarin@univ-st-etienne.fr}

\footnotetext[1]{Normandie Universit\'e,
Universit\'e de Caen, 
UMR 6139, 
Campus II, Boulevard Mar\'echal Juin, 
B.P. 5186, 
14032 Caen Cedex, France.}
\footnotetext[2]{Institut Camille Jordan, UMR 5208
Site de Saint-Etienne,
23 rue du Dr. P. Michelon,
42023 Saint-Etienne, France}
\footnotetext[3]{Supported by the ANR HAMOT}

\subjclass[2010]{Primary : 11R58, 11M38, Secondary : 11G09}

\begin{abstract}
We investigate the arithmetic of special values of a new class of $L$-functions recently introduced by the second author. We prove that these special values are encoded in some particular polynomials which we call Anderson-Stark units. We then use these Anderson-Stark units to prove that $L$-functions can be expressed as sums of polylogarithms.
\end{abstract}

\maketitle

\tableofcontents

\section{Introduction}
A major theme in the arithmetic theory of global function fields is the study of the arithmetic properties of  special values of D. Goss $L$-functions. A typical example of such a function is given by the Carlitz-Goss zeta function $\zeta_A (.),$ where $A=\mathbb F_q[\theta]$ is the polynomial ring in the variable $\theta$ with coefficients in a finite field $\mathbb F_q.$ Its special values are given by the following formula:
$$\forall n\geq 1, \, \zeta_A(n)=\sum_{a\in A_{+}} \frac{1}{a^n} \, \in K_{\infty},$$
where $A_{+}$ is the set of monic elements in $A$ and $K_{\infty}= \mathbb F_q((\frac{1}{\theta})).$
In $1990,$ G. Anderson and D. Thakur proved the following fundamental result  (\cite[ Theorem 3.8.3]{Anderson-Thakur_TensorPowers}):  for $n\geq 1,$ there exists $z_n \in {\rm Lie}(C^{\otimes n})(K_{\infty})$ such that $\exp_n(z_n) \in C^{\otimes n}(A),$ and :
$$\Gamma_n \zeta_A(n) = e_n(z_n),$$
where $\exp_n $ is the exponential map associated to the $n$th tensor power of the Carlitz module $C^{\otimes n},$  $e_n(z_n)$ is the last coordinate of $z_n \in  K_{\infty}^n,$ and $\Gamma_n\in A$ is the Carlitz factorial (we refer the reader to \cite{BrownawellPapanikolas} for the basic properties of $C^{\otimes n}$). This result has recently been generalized by M. A. Papanikolas in \cite{Papanikolas_Logalg} who proved a log-algebraicity Theorem for $C^{\otimes n}$ in the spirit of the  work of G. Anderson in \cite{Anderson_Log-algebraicity}. M. A. Papanikolas applies this log-algebraicity Theorem to obtain remarkable explicit formulas for a large class of special values of  D. Goss Dirichlet $L$-functions. Observe that the $t$-motive associated to the $t$-module  $C^{\otimes n}$ can be understood as the following object : $(A[t], \tau),$ where $t$ is an indeterminate over $K=\mathbb F_q(\theta),$ and $\tau : A[t]\rightarrow A[t]$ is the $\mathbb F_q[t]$-linear map defined as follows :
\[\tau \left(\sum_{k\geq0} a_k t^k \right)= (t-\theta)^n \left(\sum_{k\geq 0} a_k^q t^k\right),\]
where $a_k\in A.$

\bigskip

 Let $s\geq 1$ be an integer and let $t_1, \dots, t_s$ be $s$ indeterminates over $K.$  Consider the following object:  $(A[t_1, \dots, t_s], \tau)$ where $\tau : A[t_1, \dots, t_s]\rightarrow A[t_1, \dots, t_s]$ is  the morphism of $\mathbb F_q[t_1, \dots, t_s]$-modules, semi-linear\footnote{We signal here to avoid confusion that in the rest of the article $\tau$ will denote more generally a morphism semi-linear with respect to $\tau_0$.} with respect to $\tau_0 : A \to A, x\mapsto x^q$, given by :
\[\tau \left( \sum_{i_1, \dots, i_s \in \mathbb N} a_{i_1, \dots , i_s} t_1^{i_1} \cdots t_s^{i_s} \right)= 
(t_1-\theta)\cdots (t_s-\theta) \left(\sum_{i_1, \dots, i_s \in \mathbb N} a_{i_1, \dots , i_s}^q t_1^{i_1} \cdots t_s^{i_s}\right),\]
where $a_{i_1, \dots , i_s}\in A.$ Note  that we have a natural morphism of $\mathbb F_q[t_1, \dots,t_s]$-algebras : 
\[\phi : A[t_1, \dots, t_s]\rightarrow {\rm End}_{\mathbb F_q[t_1, \dots, t_s] } A[t_1, \dots, t_s]\]  
given by $\phi_{\theta}= \theta +\tau.$
Let $\mathbb T_s(K_{\infty})$ be the Tate algebra in the variables $t_1, \dots, t_s,$ with coefficients in $K_{\infty}$. Then $\tau$ extends naturally to a continuous morphism of $\mathbb F_q[t_1, \dots, t_s]$-module on $\mathbb T_s(K_{\infty}).$  
The second author introduced (see \cite{Pellarin_Lseries}, \cite{Perkins_Lseries}, \cite{Perkins_explicitformulae}, \cite{AnglesPellarin_UniversalGT}) for integers $N\in \Z$ and $s\ge 0$ the following $L$-series
\[L(N, s)=\sum_{a\in A_{+}} \frac{a(t_1)\cdots a(t_s)}{a^N},\]
which converge in $\mathbb T_s(K_{\infty})$. If $z$ is another indeterminate, we also set :
\[L(N,s,z) = \sum_{d\ge 0} z^d \sum_{a\in A_{+,d}} \frac{a(t_1)\dots a(t_s)}{a^N} \in K[t_1, \dots, t_s][[z]]\]
These series converge at $z=1$ in $\mathbb T_s(K_{\infty})$  and we have the equality :
\[L(N,s) = L(N, s, z)\mid_{z=1}.\]
Our main goal in this article is the study of the arithmetic properties of the $L(N, s, z), N\in \mathbb Z.$ Let us give a brief description of our principal results.

\medskip

We let $\tau $ act on $K[t_1, \dots, t_s][[z]]$ by :
\[\tau (\sum_{k\geq 0} f_k z^k)= \sum_{k\geq 0}\tau (f_k) z^{k},\]
where $f_k\in K[t_1, \dots, t_s]$. The exponential function associated to $\phi$ is defined by :
\[\exp_{\phi} =\sum_{i\geq 0}\frac{1}{D_i}\tau^i,\]
where $D_0=1,$ and for $i\geq 1,$ $D_i =(\theta^{q^i}-\theta ) D_{i-1}^q.$ We also set
\[\exp_{\phi,z} =\sum_{i\geq 0}\frac{z^i}{D_i}\tau^i.\]
A formulation of the $s$-variable version of Anderson's log-algebraicity Theorem is (see Theorem \ref{logalgthm} and Proposition \ref{propLsigma}):
\[\exp_{\phi,z} (L(1, s, z))\in A[t_1, \cdots, t_s, z]\]
from which we also deduce that, in $\TT_s(\Kinf)$ :
\[\exp_{\phi} (L(1, s))\in A[t_1, \cdots, t_s].\]
This $s$-variable version has been proved in \cite{AnglesPellarinTavares_LseriesTate} as a consequence of a class formula. We give here a more direct proof, close to Anderson's original proof in \cite{Anderson_Log-algebraicity}.

The special elements $\exp_{\phi,z} (L(1, s, z))$ and $\exp_{\phi} (L(1, s))$ play the role of Stark units in our context. 
Let us give an example, for $1\leq s\leq q-1,$ by Proposition  \ref{propLsigma}
we have the following equality in $\mathbb T_s :$
\[L(1, s)=\log_{\phi}(1),\]
where $\log_{\phi}=\sum_{i\geq 0}\frac{1}{l_i}\tau^i$ is the Carlitz logarithm,  $l_0=1$ and for $i\geq 1,$ $l_i=(\theta-\theta^{q^i})l_{i-1},$. We define for $N>0$, the $N$th "polylogarithm"
\[\log_{\phi,N,z}=\sum_{i\geq 0}\frac{z^i}{l_i^N}\tau^i.\]
 Set  $b_0(t)=1,$ and for $n\geq 1,$ $b_n(t)=\prod_{k=1}^{n_1}(t-\theta^{q^k}).$ Let $N\geq 1$ be an integer and let $n\geq 1$ be the unique integer such that $q^r\geq N>q^{r-1}.$ We can then prove (see Theorem \ref{polylogthm} for the precise statement) that there exists a finite set of completely explicit elements $h_j \in A[t_1, \cdots, t_s, z]$, $0\le j \le d$ that are built from the "unit" $\exp_{\phi,z} (L(1,{n+q^r-N}, z)),$ such that :
$$l_{r-1}^{ q^r-N}b_r(t_1)\cdots b_r(t_n) L(N, n, z) = \sum_{j=0}^d \theta^j \log_{\phi, N,z} (h_j).$$

\medskip

The paper is organized as follows : we first (\S 3) introduce a Banach space $\BB_s$ which is a completion of an $s$-variable polynomial ring for a norm similar to the one considered by Anderson in \cite{Anderson_Log-algebraicity}. The study of different natural Carlitz actions on $\BB_s$ allows us to endow $\BB_s$ with an action of the Tate algebra $\TT_s$ and to translate some statements on $\BB_s$ into statements on $\TT_s$. We then (\S 4) prove the $s$-variable log-algebraicity theorem, following the lines of Anderson's proof in \cite{Anderson_Log-algebraicity}, and establish some properties of the special polynomials. We also state two "converses" to the log-algebraicity theorem (Propositions \ref{convlem1} and \ref{convlem2}). In the next section (\S5) we translate the preceding results in $\TT_s$, so that the $L$-functions $L(1,s,z)$ and $L(1,s)$ appear naturally. The last section (\S 6) is devoted to the proof that $L(N,n,z)$ can be expressed as a sum of polylogarithms.

\section{Notation}
Let $\F_q$ be a finite field with $q$ elements, where $q$ is a power of a prime $p$, $\theta$ an indeterminate over $\F_q$, $A= \F_q[\theta]$, $A^*=A\backslash\{0\}$ and $K=\F_q(\theta)$. The set of monic elements (respectively of degree $j\ge 0$) of $A$ is denoted by $A_+$ (respectively $A_{+,j}$). Let $\vinf$ be the valuation on $K$ given by $\vinf(\frac a b) = \deg_\theta b-\deg_{\theta}a$. We identify with $\F_q((\frac 1 \theta))$ the completion $\Kinf$ of $K$ with respect to $\vinf$. Let $\Cinf$ be the completion of an algebraic closure of $\Kinf$. Then $\vinf$ extends uniquely to a valuation on $\Cinf$, still denoted by $\vinf$,  and we set for all $\alpha \in \Cinf$, $|\alpha|_\infty = q^{-\vinf(\alpha)}$.
The algebraic closures of $K$ and $\F_q$ in $\Cinf$ will be denoted by  $\overline K$ and $\overline {\F_q}$. 

Let $\tau$ denote an operator which we let act as the Frobenius on $\Cinf$ : for all $\alpha \in \Cinf,$ $\tau(\alpha) = \alpha^q$. If $R$ is a ring endowed with an action of $\tau$ (for instance, $R$ a subring of $\Cinf$ stable under $\tau$), then we denote by $R[[\tau]]$ the ring of formal series in $\tau$ with coefficients in $R$ subject to the commutation rule : for all $r\in R, \tau.r = \tau(r).\tau$. We also denote by $R[\tau]$ the subring of $R[[\tau]]$ of polynomials in $\tau$.

 The \emph{Carlitz module} is the unique morphism $C_{\fdot} : A \to A[\tau]$ of $\F_q$-algebras determined by $C_\theta = \theta +\tau$.  
If $M$ is an $A$-module endowed with a semi-linear endomorphism $\tau_M$ ($\forall m \in M, \forall a \in A $, $\tau_M(am) = \tau(a)\tau_M(m)$), then $C_{\fdot}$ induces a new action of $A$ on $M$ ; endowed with this action, the $A$-module $M$ is denoted by $C(M)$.

The \emph{Carlitz exponential} is the formal series
\[\exp_C  = \sum_{i\ge 0} \frac 1 {D_i} \tau^i \in K[[\tau]]\]
where $D_0=1$ and for $i\ge 1, D_i = (\theta^{q^i}-\theta)D_{i-1}^q$. The evaluation $\exp_C : \Cinf \to \Cinf$ ; $x\mapsto \exp_C(x) =   \sum_{i\ge 0} \frac 1 {D_i} \tau^i(x)$ defines an entire $\F_q$-linear function on $\C_\infty$ and $\ker (\exp_C : \Cinf \to \Cinf)= \per A$ where $\per$ is the Carlitz period defined by (see \cite[Chapter 3]{Goss_BasicStructures}):
\[\per = \sqrt[q-1]{\theta-\theta^q} \prod_{i\ge 1}\left(1 - \frac{\theta^{q^i}-\theta}{\theta^{q^{i+1}}-\theta}\right) \in \sqrt[q-1]{-\theta}(\theta + \F_q[[\frac 1 \theta]]).\]
The \emph{Carlitz logarithm} is the formal series
\[\log_C  = \sum_{i\ge 0} \frac 1 {l_i} \tau^i \in K[[\tau]]\]
where $l_0=1$ and for $i\ge 1, l_i = (\theta- \theta^{q^i})l_{i-1}$. It satisfies in $K[[\tau]]$ the equality $\log_C. \exp_C = 1$. It defines a function $x \mapsto \log_C(x)$ on $\Cinf$ converging for $\vinf(x) >-\frac q{q-1}$. Moreover, if $\vinf(x) >-\frac q{q-1}$, then $\vinf(x) = \vinf(\exp_C(x)) = \vinf(\log_C(x))$ and $\exp_C\circ \log_C(x) = x = \log_C\circ\exp_C(x)$.
We have the formal identities in $K[[\tau]]$ for all $a\in A$:
\[\exp_C a = C_a \exp_C \ \textrm{ and } \log_CC_a = a\log_C.\]
The identity $\exp_C(ax) = C_a(\exp_C(x))$ holds for all $x\in \Cinf, a \in A$.

The set of $A$-torsion points of $C(\Cinf)$ is denoted by $\Lambda_C\subset C(\overline K)$. Let $a\in A$ with $\deg_\theta a>0$, the $a$-torsion points are precisely the elements $\exp_C(\frac{b\per}a) \in \overline K$ with $b\in A$ and $\deg_\theta b<\deg_\theta a$. Therefore, $\Lambda_C = \exp_C(K\per)$. Since $\exp_C$ is continuous for the topology defined by $\vinf$, the closure of $\Lambda_C$ in $\Cinf$ is the compact set :
\[\mathfrak K = \overline{\Lambda_C} = \exp_C(\Kinf \per) = \exp_C\left(\frac 1\theta\F_q[[\frac 1 \theta]]\per\right) = \sqrt[q-1]{-\theta}\F_q[[\frac 1 \theta]]\]
where the last equality comes from the facts that $\per \in \sqrt[q-1]{-\theta}(\theta + \F_q[[\frac 1 \theta]])$, and that for $\lambda \in \F_q^*$ and $n\ge 1$, $\exp_C\left(\frac{\lambda \per}{\theta^n}\right) \equiv \frac{\lambda \per}{\theta^n} \mod \frac \per {\theta^{n+1}}\F_q[[\frac 1 \theta]]$. It is customary to consider $\Kinf \per$ as an analogue of the imaginary line ; the compact $\mathfrak K$ is then an analogue of the unit circle.
Remark that $\exp_C$ and $\log_C$ define reciprocal automorphisms of $\mathfrak K$.

\section{Some functional analysis}

\subsection{General settings}

Let $s\ge 1$ be a fixed integer and $\XX=(X_1, \dots, X_s)$ be a set of indeterminates over $\Cinf$. We want to consider polynomials $F\in \Cinf[\XX]$ as polynomial functions on $\mathfrak K^s$. Thus we introduce the following norm on $\Cinf[\XX]$ :
\[\|F\| = \max \left\{ |F(x_1, \dots, x_s)|_\infty \ ; \ x_1, \dots, x_s \in \mathfrak K  \right\}.\]
Since $\mathfrak K$ is compact and infinite, this is a well-defined, ultrametric norm of $\Cinf$-algebra. (In particular, for all $F,G\in \Cinf[\XX]$, $\|FG\|\le \|F\|\|G\|$. Moreover, $\|F\|=0 \Rightarrow F=0$ is a consequence of the fact that $\mathfrak K$ is infinite.)

If $\ii=(i_1, \dots, i_s)$ where the $i_j\ge 0$ are integers, then we write $\XX^\ii$ for $X_1^{i_1}\dots X_s^{i_s}$ and $|\ii| = i_1+\dots+i_s$.
\begin{lem}\label{lemdense}
\begin{enumerate}
\item If $\ii \in \N^s$, then $\|\XX^{\ii}\| = q^{\frac{|\ii|}{q-1}}$.
\item Write for $n\ge 1$, $\lambda_{\theta^n} = \exp_C(\frac\per{\theta^n}) \in \Lambda_C$ and let $W\subset \Lambda_C$ be the $\F_q$-vector space spanned by the $\lambda_{\theta^n}$, $n\ge 1$. Then $W$ is dense in $\mathfrak K$. In particular, for all $F\in \Cinf[\XX]$,
\[\|F\| = \sup \left\{ |F(\mathbf x)|_\infty \ ; \ \mathbf x \in \Lambda_C^s  \right\} = \sup \left\{ |F(\mathbf x)|_\infty \ ; \ \mathbf x \in W^s \right\}.\]
\end{enumerate}
\end{lem}
\begin{proof}
\begin{enumerate}
\item This is a consequence of the fact that if $a,b\in A^*$ with $\deg_\theta a<\deg_\theta b$, then
\[\vinf\left(\exp_C\left(\frac{a\per}b\right)\right) = \vinf \left(\frac{a\per}b\right) = \deg_\theta b-\deg_\theta a -\frac q{q-1}\ge \frac{-1}{q-1}.\]
\item This follows from the fact that the $\F_q$-vector space spanned by the $\frac 1 {\theta^n}$ for $n\ge 1$ is $\frac 1 \theta \F_q[\frac 1 \theta]$ which is dense in $\frac 1 \theta \F_q[[\frac 1 \theta]]$.
\end{enumerate}
\end{proof}
\begin{rmq}
Note that the norm $\|.\|$ is not multiplicative. We shall give an example in the one variable case. We have
\[\|C_\theta(X)\| = \|X\| = q^{\frac 1{q-1}}\]
but $C_\theta(X) = \prod_{\lambda\in \F_q} \left(X-\lambda\exp_C(\frac \per \theta)\right)$, where for all $\lambda \in \F_q$, $\|X-\lambda\exp_C(\frac \per \theta)\| = q^{\frac 1 {q-1}}$.
\end{rmq}

Since $\Lambda_C$ is the torsion set of $C(\Cinf)$, it is naturally endowed with the Carlitz action of $A$ : if $x\in \Lambda_C, a\in A$, then $C_a(x)\in \Lambda_C$, which extends by continuity to $\mathfrak K$. Thus, we get a natural action of the multiplicative monoid of $A$ on the polynomial functions on $\mathfrak K^s$ : 
\begin{equation}\label{action*}
\forall F(\XX) \in \Cinf[\XX], \forall a\in A \ ; \ a*F(\XX) = F(C_a(X_1), \dots, C_a(X_s)).
\end{equation}

Observe that since $\forall a\in A^*, C_a : \Lambda_C \to \Lambda_C$ is surjective, this action is isometric with respect to the norm $\| \cdot \|$.

\subsection{The one variable case}
Set $L=\Kinf(\per)$ and $\pi = \frac{\theta}\per$. Since $\vinf(\pi)= \frac 1 {q-1}$, the valuation ring of $L$ is
\[O_L = \F_q[[\pi]] = \sum_{k=0}^{q-2}\pi^k\F_q[[\frac 1 \theta]],\]
and its maximal ideal is
\[\PP_L = \pi O_L.\]
Recall that since $\Kinf = A\oplus  \frac 1 \theta \F_q[[\frac 1 \theta]]$, we have 
\[\exp_C(\per \Kinf) = \frac 1 \pi \F_q[[\frac 1 \theta]] \subset \PP_L^{-1}.\]
Let $N\in \N=\{0, 1, \dots\}$ and let $N=\sum_{i=0}^rN_iq^i$, with for all $i$, $0\le N_i \le q-1$ be its base-$q$ decomposition. Set $l_q(N) = \sum_{i=0}^r N_i$. We define the polynomial $G_N(X)$ by
\[G_N(X) = \pi^{l_q(N)} \left(\prod_{i=0}^r(\theta^i*X)^{N_i}\right) = \pi^{l_q(N)} \left(\prod_{i=0}^rC_{\theta^i}(X)^{N_i}\right) \in L[X].\]
\begin{lem}
\begin{enumerate}
\item The set $\{G_N(X), N\ge 0\}$ generates $L[X]$ as an $L$-vector space.
\item For $N\in \N$, we have : \[\|G_N(X)\| = 1.\]
\end{enumerate}
\end{lem}
\begin{proof}
\begin{enumerate}
\item It follows from the fact that $\forall N\ge 0,\deg_X(G_N(X)) = N$.
\item We remark that $\forall \lambda \in \Lambda_C$, $\vinf(G_N(\lambda))\ge 0$ and that if $\alpha\in \F_q^*$ and $\lambda_{\theta+\alpha}=\exp_C\left(\frac \per{\theta +\alpha}\right)$, then $\vinf (G_N(\lambda_{\theta+\alpha}))=0$.
\end{enumerate}
\end{proof}
If $\beta=(\beta_i)_{i\ge 1}$ is a sequence of elements in $\F_q$, we set :
\[\lambda(\beta) = \sum_{i\ge 1} \beta_i\exp_C\left(\frac \per {\theta^i}\right) \in \PP_L^{-1}.\]
Remark that if we set $\mu(\beta) = \sum_{i\ge 1} \frac {\beta_i} {\theta^i} \in \Kinf$, then we have $\lambda(\beta) = \exp_C(\per \mu(\beta))$.
\begin{lem}\label{lemBanachbase1}
Let $\beta=(\beta_i)_{i\ge 1}$ be a sequence of elements in $\F_q$, and $N=\sum_{i=0}^rN_iq^i$ be a non negative integer written in base $q$. Then 
\[G_N(\lambda(\beta)) \equiv \prod_{i=0}^r \beta_{i+1}^{N_i} \mod \PP_L.\]
\end{lem}
\begin{proof}
Observe that 
\[\exp_C\left(\frac \per \theta\right) \equiv \frac 1 \pi \mod O_L.\]
Thus, for $j\ge 0$,
\[\pi C_{\theta_j}(\lambda(\beta))\equiv \beta_{j+1} \mod \PP_L.\]
Whence the result.
\end{proof}
\begin{lem}\label{lemBanachbase2}
Let $k,r$ be two integers such that $r\ge 1$ and $1\le k \le q^r$, let $\alpha_1, \dots, \alpha_k\in \F_q^*$ and let $N_1, \dots, N_k$ be $k$ distinct integers in $\{0, \dots, q^r-1\}$. Write $N_i = \sum_{j=0}^{r-1}n_{i,j} q^i$ in base $q$. Then, there exists $\beta_1, \dots, \beta_r \in \F_q$ such that 
\[\sum_{i=1}^k \alpha_i \prod_{j=0}^{r-1}\beta_{j+1}^{n_{i,j}} \neq 0\]
with the convention that $0^0=1$.
\end{lem}
\begin{proof}
We proceed by induction on $r$.

If $r=1$, then $k\le q$ and for $1\le j \le k$, $N_i=n_{i,1} \in \{0, \dots, q-1\}$. Since the $N_i$'s are distinct, the polynomial $\sum_{i=1}^k\alpha_iX^{N_i}$ is not divisible by $X^q-X$, and this implies the assertion of the lemma in this case.

We assume now that the lemma is proved for all integers less than $r-1\ge 1$, and we also assume that at least one $N_i$ is $\ge q^{r-1}$.
We define an equivalence relation over the set $\{1, \dots, k\}$: for all $1\le i, i'\le k$, $i\sim i'$ if and only if $n_{i,j} = n_{i',j}$ for all $1\le j \le r-2$ (that is, if $N_i \equiv N_{i'} \mod q^{r-1}$). We denote by $I_1, \dots, I_t$ the equivalence classes and if $i \in I_m$ we define for $1\le j \le r-2$, $n_j^{(m)} = n_{i,j}$ the common value.
Let $\beta_1, \dots, \beta_r\in \F_q^*$, then
\[ \sum_{i=1}^k \alpha_i \prod_{j=0}^{r-1}\beta_{j+1}^{n_{i,j}} = \sum_{m=1}^t \left(\sum_{i\in I_m} \alpha_i \beta_r^{n_{i,r-1}}\right) \prod_{j=0}^{r-2}\beta_{j+1}^{n_{j}^{(m)}}.\]
Now, by the case $r=1$, we can find $\beta_r$ such that the sum $\sum_{i\in I_1} \alpha_i \beta_r^{n_{i,r-1}}$ is not zero and we can apply the induction hypothesis to conclude the proof.
\end{proof}

Let $\Kinf \subseteq E\subseteq \Cinf$ be a field complete with respect to $|\cdot |_\infty$, and $\BB(E)$ denote the completion of $E[X]$ with respect to $\| \cdot \|$.

\begin{thm}\label{thmBanachbase_1}
The family $\{G_N(X), N\ge 0\}$ forms an \emph{orthonormal basis} of the $E$-Banach space $\BB(E)$, that is :
\begin{enumerate}[(i)]
\item Any $F\in \BB(E)$ can be written in a unique way as a convergent series $F=\sum_{N\ge 0} f_N G_N(X)$ with $f_N\in E, N\ge0$ ; $\lim_{N\to \infty}f_N =0$ ;
\item if $F$ is written as above, $\|F\|=\max_{N\ge 0} |f_N|_\infty$.
\end{enumerate}
\end{thm}
\begin{proof}
It is enough to prove the above properties $(i)$ and $(ii)$ for $F\in E[X]$. Note that property $(i)$ is a consequence of the fact that $\deg_X G_N(X) = N$ for all $N\ge 0$. Let us prove property $(ii)$. It is enough to consider $F=\sum_{i=0}^r x_r G_{N_i}$ with for all $0\le i \le r$, $\vinf(x_i)=0$. We are reduced to proving that $\|F\|=1$, that is, $\|F\|\ge 1$ since we already know the converse inequality. And the existence of $\lambda \in \Lambda_C$ such that $\vinf(F(\lambda))=0$ is a consequence of Lemmas \ref{lemBanachbase1} and \ref{lemBanachbase2}.
\end{proof}
\subsection{The multivariable case}
Let $s\ge 1$ be an integer, we define for a field $\Kinf \subseteq E\subseteq \Cinf$ complete with respect to $| \cdot |_\infty$, $\BB_s(E)$ to be the completion of $E[\XX]$ with respect to $\| \cdot \|$. We write also for short $\BB_s = \BB_s(\Cinf)$. Observe that for $N_1, \dots, N_s \in \N$, we have :
\[\left\|G_{N_1}(X_1) \cdots G_{N_s}(X_s)  \right\| =1.\]
\begin{thm}\label{thmBanachbase_s}
Let $L\subseteq E\subseteq \Cinf$ be complete with respect to $|\cdot|_\infty$, then the family 
\[\{G_{N_1}(X_1) \cdots G_{N_s}(X_s) , N_1, \dots, N_s \in \N\}\]
 forms an orthonormal basis of the $E$-Banach space $\BB_s(E)$, that is :
\begin{enumerate}[(i)]
\item Any $F\in \BB_s(E)$ can be written in a unique way as the sum of a summable family 
\[F=\sum_{(N_1, \dots, N_s) \in \N^s} f_{N_1, \dots, N_s} G_{N_1}(X_1) \cdots G_{N_s}(X_s)\]
 with $\forall N_1, \dots, N_s \in \N, f_{N_1, \dots, N_s}\in E$ ; $f_{N_1, \dots, N_s} $ goes to $0$ with respect to the Fr\'echet filter  ;
\item if $F$ is written as above, $\|F\|=\max\{|f_{N_1, \dots, N_s}|_\infty, N_1, \dots, N_s \in \N\}$.
\end{enumerate}
\end{thm}
\begin{proof}
We proceed by induction on $s\ge 1$. The case $s=1$ is the statement of Theorem \ref{thmBanachbase_1}.
Assume now that $s\ge 2$ and that the theorem is true for $s-1$. It will be enough to prove $(i)$ and $(ii)$ for polynomials, and $(i)$ is still an easy consequence of  $\deg_X G_N(X) = N$ for all $N\ge 0$. Write a polynomial
\[F=\sum_{i=0}^r \alpha_r G_{N_i}(X_s) \in E[\XX], \text{ where } \forall 1\le i \le r, \alpha_i \in E[X_1, \dots, X_{s-1}].\]
Write for $1\le i \le r$, the polynomial
\[\alpha_i = \sum_{i_1, \dots, i_{s-1}} \alpha_{i_1, \dots, i_{s-1}}^{(i)} G_{i_1}(X_1)\cdots G_{i_{s-1}}(X_s)\] 
with $\alpha_{i_1, \dots, i_{s-1}}^{(i)} \in E$. Then the induction hypothesis tells that for all $i$ :
\[ \|\alpha_i\| = \max\left\{\left|\alpha_{i_1, \dots, i_{s-1}}^{(i)} \right|_{\infty}, i_1, \dots, i_{s-1} \in \N\right\}.\]
Thus
\[\|F\| \le \max_{1\le i \le r} \|\alpha_i\| = \max \left\{\left|\alpha_{i_1, \dots, i_{s-1}}^{(i)} \right|_{\infty}, i_1, \dots, i_{s-1}, i \in \N \right\}.\]
Let $1\le i_0\le r$ be such that $\|\alpha_{i_0}\| = \max_{1\le i \le r} \|\alpha_i\|$, to prove the converse inequality, we will find $\lambda_1, \dots, \lambda_s\in \mathfrak K^s$ such that $\left|F(\lambda_1, \dots, \lambda_s )\right|_\infty = \|\alpha_{i_0}\|$. Let $\lambda_1, \dots, \lambda_{s-1}\in \mathfrak K^{s-1}$ such that 
\[\left|\alpha_{i_0}(\lambda_1, \dots, \lambda_{s-1})\right|_\infty = \|\alpha_{i_0}\|.\]
Then, by the case $s=1$, 
\[\|F(\lambda_1, \dots, \lambda_{s-1},X_s)\|= \max \left|\alpha_{i_0}(\lambda_1, \dots, \lambda_{s-1})\right|_\infty = \|\alpha_{i_0}\|\]
Therefore, we can find $\lambda\in \mathfrak K$ such that $ \left|F(\lambda_1, \dots, \lambda_{s-1}, \lambda )\right|_\infty = \|\alpha_{i_0}\|$, proving that $\|F\|=\|\alpha_{i_0}\|$ and the theorem.
\end{proof}
For all $N = \sum_{i=0}^rN_iq^i\ge 0$, define
\[H_N(X) =  \left(\prod_{i=0}^r(\theta^i*X)^{N_i}\right) = \pi^{-l_q(N)}G_N(X)  \in \Kinf[X].\]
Then the $H_N$'s generate $\Kinf[X]$ and $\|H_N(X)\| =  q^{\frac{l_q(N)}{q-1}}$. 
If $E$ does not contain $L$, in particular if $E=\Kinf$, then $G_N$ has no longer coefficients in $E$ and there might not exist an orthonormal basis of $\BB(E)$. However, Theorem \ref{thmBanachbase_s} still implies :
\begin{cor}\label{orthogbase_s}
Let $\Kinf\subseteq E\subseteq \Cinf$ be complete with respect to $|\cdot|_\infty$, then the family 
\[\{H_{N_1}(X_1) \cdots H_{N_s}(X_s) , N_1, \dots, N_s \in \N\}\]
 forms an \emph{orthogonal basis} of the $E$-Banach space $\BB_s(E)$, that is :
\begin{enumerate}[(i)]
\item Any $F\in \BB_s(E)$ can be written in a unique way as the sum of a summable family 
\[F=\sum_{(N_1, \dots, N_s) \in \N^s} f_{N_1, \dots, N_s} H_{N_1}(X_1) \cdots H_{N_s}(X_s)\]
 with $\forall N_1, \dots, N_s \in \N, f_{N_1, \dots, N_s}\in E$ ; $|f_{N_1, \dots, N_s}|_\infty q^{{\frac{l_q(N_1)+\cdots+l_q(N_s)}{q-1}}}$ goes to $0$ with respect to the Fr\'echet filter  ;
\item if $F$ is written as above, 
\begin{eqnarray*}
\|F\|&=&\max\{|f_{N_1, \dots, N_s}H_{N_1}(X_1) \cdots H_{N_s}(X_s)|_\infty, N_1, \dots, N_s \in \N\}\\
&=&\max\{|f_{N_1, \dots,  N_s}|_\infty q^{{\frac{l_q(N_1)+\cdots+l_q(N_s)}{q-1}}}, N_1, \dots, N_s \in \N\}.
\end{eqnarray*}
\end{enumerate}
\end{cor}
\subsection{The Carlitz action}
In this section, $\Kinf\subseteq E\subseteq \Cinf$ is a field complete with respect to $|\cdot|_\infty$. 
Note that the action $*$ of $A$ on $E[\XX]$ defined in \eqref{action*} satisfies that for all $a\in A^*$, the map $F\mapsto a*F$ is an isometry on $E[\XX]$. Thus, the action $*$
extends to an action, still denoted $*$, of $A$ on $\BB_s(E)$, such that for all $a\in A^*$, the map $F\mapsto a*F$ is an isometry on $\BB_s(E)$.

Now, instead of considering the simultaneous action of $A$ on each of the $X_j$, we will separate this action into actions on a single variable $X_j$, namely, for $1\le j \le s$, $F\in \BB_s(E)$ and $a\in A$, we set :
\begin{equation}\label{action*j}
a*_j F(\XX) = F(X_1, \dots, X_{j-1}, C_a(X_j), X_{j+1}, \dots, X_s)
\end{equation}
This is still an action of monoid, but if we restrict this action to the set of polynomials in $E[\XX]$ which are $\F_q$-linear in the variable $X_j$, the action $*_j$ induces a structure of $A$-module. Thus we define :
\[ E[\XX]^{\lin} = \left\{F\in E[\XX] ; F \text{ is linear with respect to each of the variables } X_1, \dots, X_s \right\}\]
which is just the sub-$E$-vector space of $E[\XX]$ spanned by the monomials $X_1^{q^{i_1}}\cdots X_s^{q^{i_s}}, i_1, \dots, i_s\in \N$. 
Since the actions $*_j$ and $*_i$ commute and commute with the linear action of $E$, $E[\XX]^{\lin}$ has a structure of module over $E\otimes_{\F_q}A^{\otimes s}$, that is, if $t_1, \dots, t_s$ are new indeterminates, we identify $E\otimes_{\F_q}A^{\otimes s}$ with $E[t_1, \dots, t_s]$ and $E[\XX]^{\lin}$ has a structure of $E[t_1, \dots, t_s]$-module given by :
\begin{equation}\label{actiontj}
\forall 1\le j \le s, \ \  t_j . F(X_1, \dots, X_s) = F(X_1, \dots, X_{j-1}, C_\theta(X_j), X_{j+1}, \dots, X_s).
\end{equation}
We write $\ttt$ for the set of variable ${t_1, \dots, t_s}$ and if $\ii=(i_1, \dots, i_s)\in \N^s$, $\ttt^\ii = t_1^{i_1}\cdots t_s^{i_s}$.

The action defined by formula \ref{actiontj} extends to an action on $E[\XX]$, turning $E[\XX]$ into an $E[\ttt]$-algebra. 
We define the subordinate norm $\|.\|_\infty$ on $E[\ttt]$ by :
\[ \|f\|_\infty =\sup_{F\in E[\XX]\backslash \{0\}}\frac{\|f.F\|}{\|F\|}.\]
\begin{lem}\label{lemNormGauss}
Let $f\in E[\ttt]$, $f=\sum_{\ii} f_{\ii} \ttt^\ii$, then for all $F\in E[\XX]\backslash \{0\}$,
\[\|f\|_\infty = \max_\ii |f_{\ii}|_\infty =\frac{\|f.F\|}{\|F\|}.\]
\end{lem}
\begin{rmq}
The lemma says in particular that the norm $\|\cdot\|_\infty$ coincides with the Gauss norm on $E[\ttt]$, which is known to be multiplicative. This property also follows easily from the lemma.
\end{rmq}
\begin{proof}[Proof of the lemma]
Write $F=\sum_{N_1, \dots, N_s} F_{N_1, \dots, N_s}H_{N_1}(X_1)\cdots H_{N_s}(X_s)$ and $M= \max_\ii |f_{\ii}|_\infty$. Remark that for all $N\ge 1$ and for all $1\le i \le s$, $t_i.H_{N}(X_i) = H_{qN}(X_i)$. Since $l_q(N)= l_q(qN)$, we deduce from Corollary \ref{orthogbase_s} that $\|f.F\| \le M\|F\|$. 

Conversely, consider $(N_{1,0}, \dots, N_{s,0})$ the index, minimal for the lexicographic ordering on $\N^s$, such that 
\[|F_{N_{1,0}, \dots, N_{s,0}}|_\infty q^{\frac{l_q(N_{1,0})+\cdots + l_q(N_{s,0})}{q-1}} = \|F\|\]
and $\ii_0 = (i_{1,0}, \dots, i_{s,0})$ the index, minimal for the lexicographic ordering on $\N^s$, such that $M=|f_{\ii_0}|_\infty$.
Then, the coefficient of 
\[H_{q^{i_{1,0}}N_{1,0}}(X_1) \cdots H_{q^{i_{s,0}}N_{s,0}}(X_s)\]
in the expansion of $f.F$ in the basis of Corollary \ref{orthogbase_s} is equal to 
\[ f_{\ii_0}F_{N_{1,0}, \dots, N_{s,0}} \ + \ \text{ terms of lower norm,}\]
whence the result.
\end{proof}
We define 
\begin{itemize}
\item $\BB_s^{\lin}(E)$ the adherence of $E[\XX]^{\lin}$ in $\BB_s(E)$,
\item $\TT_s(E)$ the completion of $E[\ttt]$ for the Gauss norm $\|.\|_\infty$.
\end{itemize}
Recall that $\TT_s(E)$ is the standard Tate algebra in $s$ variables over $E$ (see \cite[\S II.1.]{FresnelVanderPut_RigidGeometry}), that is, the algebra of formal series $\sum_{\ii\in \N^s} f_\ii \ttt^\ii$ with $f_\ii\in E$ going to zero with respect to the Fr\'echet filter.
The action of $E[\ttt]$ extends naturally to an action of $\TT_s(E)$ on $\BB_s(E)$ and on $\BB_s^{\lin}(E)$.
\begin{lem}\label{isomElin}
\begin{enumerate}
\item The family $\{H_{q^{n_1}}(X_1) \cdots H_{q^{n_s}}(X_s) , n_1, \dots, n_s \in \N\}$ forms an orthogonal basis of elements of the same norm $q^{\frac s {q-1}}$ of the $E$-Banach space $\BB_s^{\lin}(E)$,
\item The map 
$\left\{\begin{array}{ccc}
\TT_s(E) & \rightarrow & \BB_s(E)\\
f & \mapsto & f.(X_1\cdots X_s)
\end{array}\right.$
is injective, with for all $f\in \TT_s(E)$, 
\[\|f.(X_1\cdots X_s)\| = q^{\frac s{q-1}} \|f\|_\infty.\]
\item $E[\XX]^{\lin} = E[\ttt] . X_1\cdots X_s$,
\item $\BB_s^{\lin}(E) = \TT_s(E) . X_1\cdots X_s$.
\end{enumerate}
\end{lem}
\begin{proof}
Since for all $1\le i \le s$ and all $n\ge 0$, $H_{q^{n}}(X_i)$ is an $\F_q$-linear polynomial of degree $q^{n}$, the family $\{H_{q^{n_1}}(X_1) \cdots H_{q^{n_s}}(X_s) , n_1, \dots, n_s \in \N\}$ forms a basis of $E[\XX]^{\lin}$ and the first assertion follows from Corollary \ref{orthogbase_s}. The relation $t_i^n.X_i = H_{q^n}(X_i)$ then implies the other assertions.
\end{proof}
As a consequence, the map $f\mapsto f.X_1\cdots X_s$ defines, up to the normalisation constant $q^{\frac s{q-1}}$, an isometric immersion of $\TT_s(E)$ into $\BB_s(E)$.
Write $A[\XX]^{\lin}=A[\XX]\cap E[\XX]^{\lin}$, we have :
\begin{lem}\label{isomAlin}
Let $f\in E[\ttt]$, then $f.(X_1\cdots X_s) \in A[\XX]^{\lin}$ if, and only if, $f\in A[\ttt]$. In particular, $A[\XX]^{\lin} = A[\ttt] . X_1\cdots X_s$.
\end{lem}
\begin{proof}
It is clear that if $f\in A[\ttt]$, then $f.(X_1\cdots X_s) \in A[\XX]$. Note that, since 
\[t_1^{i_1}\cdots t_s^{i_s} . H_{N_1}(X_1)\cdots H_{N_s}(X_s) = H_{q^{i_1}N_1}(X_1)\cdots H_{q^{i_s}N_s}(X_s),\]
a consequence of Corollary \ref{orthogbase_s} is that $\BB_s(E)$ is a torsion-free $\TT_s(E)$-module.
Then, the converse is an easy consequence of the fact that $t_i.X_i$ is a monic polynomial in $A[X_i]$.
\end{proof}

\section{Multivariable $\log$-algebraicity}
\subsection{The $\log$-algebraicity theorem}
Let $Z$ be another indeterminate over $\Cinf$. We let $\tau$ act on $\Cinf[\XX][[Z]]$ (or in the one variable case on $\Cinf[X][[Z]]$) via $\tau(F) = F^q$.

Let $F\in A[X]$ ; we form the series
\[\sum_{d\ge 0} Z^{q^d} \sum_{a\in A_{+,d}} \frac{a*F}a\in K[X][[Z]]\]
 and take $\exp_C$ of this series which makes sense in $K[X][[Z]]$. Anderson's $\log$-algebraicity theorem \cite[Theorem 3]{Anderson_Log-algebraicity} for $A$ states then 
\begin{thm}[Anderson]\label{Anderson}
For all $F \in A[X]$ 
\[ \exp_C\left(\sum_{d\ge 0} Z^{q^d} \sum_{a\in A_{+,d}} \frac{a*F}a \right) \in A[X,Z].\]
\end{thm}
The aim of this section is to give a multivariable generalisation of this result. But first, let us give a simple proof of Theorem \ref{Anderson} in the case of $F=X$ and $Z=1$.
\begin{exe}\label{exlogalg1}
Write $X = \exp_C Y$, where $Y=\log_C X \in K[[X]]$. Then $a*X = \exp_C(aY) = \sum_{j\ge 0} \frac{a^{q^j}Y^{q^j}}{D_j}.$ Thus,
\[\sum_{d\ge 0} Z^{q^d}\sum_{a\in A_{+,d}} \frac{a*X}a = \sum_{d\ge 0} Z^{q^d}\sum_{a\in A_{+,d}} \sum_{j\ge 0} \frac{a^{q^j-1}Y^{q^j}}{D_j} = \sum_{j\ge 0} \frac{Y ^{q^j}}{D_j}\sum_{d\ge 0} Z^{q^d} \sum_{a\in A_{+,d}}a^{q^j-1}.\]
But one can evaluate at $Z=1$ since  (see \cite[Example 8.13.9]{Goss_BasicStructures}) $\sum_{a\in A_{+,d}}a^{q^j-1}=0$ for $d\gg j$, and moreover $\sum_{d\ge 0} \sum_{a\in A_{+,d}}a^{q^j-1}=0$ for all $j>0$ while this sum equals $1$ when $j=0$. Therefore, we get :

\[\sum_{d\ge 0} \sum_{a\in A_{+,d}} \frac{a*X}a =Y = \log_C X.\]
\end{exe}

\begin{lem}\label{logalglem1}
If $F\in A[\XX]$ satisfies $\|F\|\le 1$, then $F\in \F_q$.
\end{lem}
\begin{proof}
If $\lambda_1, \dots, \lambda_s \in \Lambda_C$, then $F(\lambda_1, \dots, \lambda_s)$ is integral over $A$, and the condition $\|F\|\le 1$ implies that for all $\lambda_1, \dots, \lambda_s \in \Lambda_C$, $F(\lambda_1, \dots, \lambda_s) \in \overline{\F}_q$. But $\F_q$ is algebraically closed in $K(\lambda_1, \dots, \lambda_s)$(see \cite[Corollary to Theorem 12.14]{Rosen_NumbertheoryFunctionfields}), so that $F(\lambda_1, \dots, \lambda_s) \in \F_q$. Now, for any $\lambda_1, \dots, \lambda_{s-1} \in \Lambda_C$, the polynomial $F(\lambda_1, \dots, \lambda_{s-1}, X_s)$  takes at least one value infinitely many times. An easy induction on $s$ then implies that $F$ is constant, that is $F\in \F_q$.
\end{proof}
We define an action of the multiplicative monoid $A^*$ over $\Cinf[\XX][[Z]]$ by letting for $F(\XX, Z) \in \Cinf[\XX][[Z]]$ and $a\in A^*$ :
\[ a*F = F\left(C_a(X_1), \dots, C_a(X_s), Z^{q^{\deg_\theta a}}\right).\]
Observe that $\exp_C$ gives rise to a well-defined endomorphism of $\Cinf[\XX][[Z]]$ and that 
\[\exp_C(K[\XX][[Z]])\subset K[\XX][[Z]].\]
Let  $F\in \Cinf[\XX]$ ; following Anderson, we set for $k<0$ :
\[L_k(F)=Z_k(F)=0\]
and for $k\ge 0$ :
\begin{eqnarray*}
L_{k}(F) &=& \sum_{a\in A_{+,k}} \frac{a*F}a \in \Cinf[\XX], \\
Z_{k}(F) &=& \sum_{j\ge 0} \frac{L_{k-j}(F)^{q^j}}{D_j}\in \Cinf[\XX].
\end{eqnarray*}
Define moreover :
\begin{eqnarray*}
l(F,Z)  &=& \sum_{a\in A_+}  \frac{a*(FZ)}a = \sum_{k\ge 0} Z^{q^k}L_k(F) \in \Cinf[\XX][[Z]],\\
\mathfrak L(F,Z) &=&  \exp_C\left(l(F,Z)\right) = \sum_{k\ge 0} Z_{k}(F) Z^{q^k} \in \Cinf[\XX][[Z]].
\end{eqnarray*}
\begin{lem}\label{logalglem2}
Let $F\in \Cinf[\XX]$ and $k\ge 0$.
\begin{enumerate}
\item $\|L_k(F)\|\le \|F\|q^{-k}$,
\item $\|Z_k(F)\| \le \max_{0\le j \le k} \|F\|^{q^j}q^{-kq^j}$.
\end{enumerate}
\end{lem}
\begin{proof}
This comes from the definitions and the fact that for all $a\in A^*$, $\|a*F\|=\|F\|.$
\end{proof}
We call a monic irreducible polynomial of $A$ a \emph{prime} of $A$.
Let $P$ be a prime of $A.$ Let $F\in K[\XX]$ and let $I$ be a finite subset of $\mathbb N^s$  such that  $F=\sum_{\ii\in I }\alpha_{\ii} \XX^{\ii}\in K[\XX],$ Let $v_P$ be the $P$-adic valuation on $K$ normalized by $v_P(P)=1$, we set :
\[v_P(F) =\inf\{ v_P(\alpha_{\ii}), \ii\in I \}.\]
Recall that we have for $F,G\in K[\XX]$, and $\lambda \in K$ :
\begin{itemize}
\item $v_P(F+G)\ge \inf(v_P(F), v_P(G))$, $v_P(FG)=v_P(F)+v_P(G)$,
\item $v_p(\lambda F) =v_P(\lambda) +v_P(F)$,
\item $v_P(F)=+\infty $ if and only if $F=0.$
\end{itemize}

\begin{lem}\label{logalglem3}
Let $P$ be a prime of $A$. Let $F\in K[\XX]$ be such that $v_P(F)\geq 0$. Then for all $k\ge 0,$ $v_P(Z_k(F))\ge 0$.
\end{lem}
\begin{proof} 
The proof is essentially the same as \cite[Proposition 6]{Anderson_Log-algebraicity}. We recall it because some details will be needed in the proof of Proposition \ref{ConvAlg1}.

Set $A_{(P)}=\{ x\in K, v_P(x)\ge 0\}.$ Let $d$ be the degree of $P$, we have in $A[\tau]$ : $C_P \equiv \tau^d \mod PA[\tau]$. We prove that if $G=\sum_{k\ge 0} G_kZ^{q^k}\in K[\XX][[Z]]$ satisfies 
\[(C_P-P*)(G)\in PA_{(P)}[\XX][[Z]],\]
then $\forall k\ge 0, G_k\in A_{(P)}[\XX]$. Set $G_k=0$ if $k<0$ and write $C_P = \sum_{i=0}^d [P]_i\tau^i$ where $[P]_0=P$, $[P]_i\in PA$ if $i<d$ and $[P]_d=1$. We have $(C_P-P*)(G) = \sum_{k\ge 0} H_k Z^{q^k}$ with for all $k\ge 0$,
\[H_k = \sum_{i=0}^d [P]_i \tau^i(G_{k-i}) - P*G_{k-d} = PG_k + \sum_{i=1}^{d-1} [P]_i \tau^i(G_{k-i}) + \tau^d(G_{k-d})- P*G_{k-d} \in PA_{(P)}[\XX].\]
In particular, $H_0 = PG_0 \in PA_{(P)}[\XX]$ so that $G_0\in A_{(P)}[\XX]$. Now, by induction on $k$, if we know that $G_{k-i}\in A_{(P)}[\XX]$ for $i=1, \dots, d$, then $\tau^d(G_{k-d})- P*G_{k-d} \in PA_{(P)}[\XX]$ and we deduce that $G_k \in PA_{(P)}[\XX]$.

Define $l^*(F,Z) =\sum_{a\in A_+, P\nmid a}  \frac{a*(FZ)}a\in A_{(P)}[\XX][[Z]]$, we have 
\begin{eqnarray*}
l(F,Z) &=& \sum_{a\in A_+, P\mid a}  \frac{a*(FZ)}a + \sum_{a\in A_+, P\nmid a}  \frac{a*(FZ)}a  \\
&=& \sum_{a\in A_+}  \frac{(aP)*(FZ)}{aP} +l^*(F,Z)  = \frac 1 P\left( P*l(F,Z)\right) +l^*(F,Z)
\end{eqnarray*}
which yields the relation :
\[Pl(F,Z)-P*l(F,Z)=Pl^*(F,Z).\]
Remark that the action $*$ commutes with $\tau$, and thus with $\exp_c$, thus if we apply $\exp_C$, we get
\[(C_P-P*)\left(\mathfrak L(F,Z)\right) = \exp_C(Pl^*(F,Z)) = \sum_{j\ge 0} \frac {P^{q^j}}{D_j}l^*(F,Z)^{q^j}\]
and since $\forall j\ge 0$, $v_P(\frac {P^{q^j}}{D_j})\ge 1$, we get $(C_P-P*)\left(\mathfrak L(F,Z)\right)\in PA_{(P)}[\XX][[Z]]$ whence $\mathfrak L(F,Z)\in A_{(P)}[\XX][[Z]]$.
\end{proof}
We can now state and prove the multivariable $\log$-algebraicity theorem :
\begin{thm}\label{logalgthm}
Let $F\in A[\XX]$, then 
\[\mathfrak L(F,Z) =  \exp_C\left(\sum_{a\in A_+}  \frac{a*(FZ)}a\right) \in A[\XX,Z].\]
\end{thm}
\begin{proof}
By Lemma \ref{logalglem3}, for all $k\ge 0$, $Z_k(F) \in A[\XX]$. If $k_0\ge 0$ is the smallest integer such that $\|F\|\le q^{k_0}$, then by Lemma \ref{logalglem2},  for all $k>k_0$, $\|Z_k(F)\|<1$. Therefore, Lemma \ref{logalglem1} tells that $Z_{k_0}(F) \in \F_q$ and for all $k>k_0$, $Z_k(F)=0$.
\end{proof}
The previous theorem can also be obtained as a consequence of a class formula for a Drinfeld module on a Tate algebra (see \cite{AnglesPellarinTavares_LseriesTate}).
\subsection{The special polynomials}
If $s\ge 1$ is an integer, we define the special polynomial :
\[\Ss_s = \Ss_s(\XX,Z) = \mathfrak L(X_1\cdots X_s,Z) \in A[\XX,Z].\]

Recall that Anderson's special polynomials are the one variable polynomials $S_m(X,Z) = \mathfrak L (X^m,Z)$. We recover $S_m(X,Z)$ from $\Ss_m$ by specializing each of the $X_j, 1\le j \le m$, to $X$. We establish in this section some properties of the polynomials $\Ss_s$.

The following proposition is used to compute explicitly the polynomial $\mathfrak L(F,Z)$.
\begin{prop}\label{propSs}
\begin{enumerate}
\item The polynomial $\Ss_s(\XX, Z)$ is $\F_q$-linear with respect to each of the variables $X_1, \dots, X_s, Z$; in particular, it is divisible by $X_1 \cdots X_sZ$.
\item If $r \in \{1, \dots, q-1\}$ satisfies $s\equiv r \mod q-1$, then :
\[\deg_Z \Ss_s \le q^{\frac{s-r}{q-1}}.\]
In particular, if $1\le s \le q-1$, we have :
\[\Ss_s = X_1\cdots X_s Z.\]
\end{enumerate}
\end{prop}
\begin{proof}
The first assertion is obvious.
By Lemmas  \ref{logalglem1}, \ref{logalglem2} and  \ref{logalglem3}, $Z_k(X_1\cdots X_s) \in \F_q$ if $k\ge \frac s{q-1}$. But since $X_1 \cdots X_s$ divides $Z_k(X_1\cdots X_s)$, we get $Z_k(X_1\cdots X_s)=0$ for $k\ge \frac s{q-1}$. The last part comes from the congruence
\[ \Ss_s \equiv X_1\cdots X_s Z \mod Z^q.\]
\end{proof}
\begin{cor}
Let $s, k_1, \dots, k_s \ge 1$ be integers such that $\sum_{j=1}^sk_j\le q-1$ and let  $a_{1,1}, \dots, a_{1,k_1}$, $\dots, a_{s,1}$, $\dots, a_{s,k_s} \in A$.
Set
\[G = (a_{1,1}*X_1) \cdots (a_{1,k_1}*X_1)\cdots (a_{s,1}*X_s) \cdots (a_{s,k_s} *X_s) \in A[\XX].\]
Then 
\[\mathfrak L(G,Z) = GZ.\]
\end{cor}
\begin{proof}
It is sufficient to consider the case where $k_j=1$ for all $1\le j \le s$ since we obtain the general case by specializing variables.
The action $*_j$ of $A$ (defined in \eqref{action*j}) satisfies for all $a\in A$, $F\in A[\XX]$
\[a*_j\left(\mathfrak L(F, Z)\right) = \mathfrak L(a*_jF,Z).\]
The corollary follows then from the relation $\mathfrak L(X_1\cdots X_s,Z) = X_1\cdots X_sZ$ since $s\le q-1$.
\end{proof}
Any $\F_q$-linear combination $F$ of polynomials of the above form still satisfies $\mathfrak L(F,Z) = FZ.$ We can ask whether there are other polynomials satisfying this relation. In fact, Proposition \ref{ConvAlg2} below assures that if  $\mathfrak L(F,Z) = FZ$, then $F \in A[\XX]$, so we can ask more generally :
\begin{qst} Describe the set of the $F\in A[\XX]$ such that $\mathfrak L(F,Z) = FZ$.
\end{qst}
\begin{lem}\label{nulSs}
Let $s\ge 1$, then $\Ss_s(\XX, 1) = 0$ if, and only if, $s\ge 2$ and $s\equiv 1 \mod q-1$.
\end{lem}
\begin{proof}
First, suppose $s\ge 2$ and $s\equiv 1 \mod q-1$.

Let $a\in A$. Recall (see \cite[Section 3.4]{Anderson-Thakur_TensorPowers}) that from the relation $C_a(X) = \exp_C(a\log_C(X))$, we deduce that we can write 
\[C_a(X) = \sum_{k=0}^{\deg_\theta a} \psi_k(a) X^{q^k}\]
where $\psi_k(x) \in A[x]$ is an $\F_q$-linear polynomial of degree $q^k$, which vanishes exactly at the polynomials $x\in A$ of degree less than $k$.
Thus
\[ a* (X_1\cdots X_s) = \sum_{k_1, \dots, k_s \ge 0} \psi_{k_1}(a)\cdots \psi_{k_s}(a) X_1^{q^{k_1}}\cdots X_s^{q^{k_s}}\]
where the right hand side is a finite sum.

We deduce that 
\[\Ss_s(\XX, Z) = \sum_{n\ge 0} Z^{q^n} \sum_{d=0}^n D_{n-d}^{-1}\sum_{k_1, \dots, k_s \ge 0} \sum_{a\in A_{+,d}} \left( \frac{ \psi_{k_1}(a)\cdots\psi_{k_s}(a)}{a} \right)^{q^{n-d}}X_1^{q^{k_1+n-d}}\cdots X_s^{q^{k_s+n-d}}\]
and by Theorem \ref{logalgthm}, this is a polynomial. Remark now that $\sum_{a\in A_{+,d}} \frac{ \psi_{k_1}(a)\cdots\psi_{k_s}(a)}{a}$ is a linear combination (with coefficients depending only on $k_1, \dots, k_s, r_1, \dots, r_s$ and independent on $d$) of sums of the form $\sum_{a\in A_{+,d}} a^{q^{r_1}+\cdots+q^{r_s}-1}$ with, for all $1\le j \le s$, $0\le r_j\le k_j$. According to  \cite[Lemma 8.8.1]{Goss_BasicStructures}, this sum vanishes for $d>\frac {q^{r_1}+\cdots+q^{r_s}-1}{q-1}$. Thus the coefficient of $X_1^{q^{m_1}}\cdots X_s^{q^{m_s}}$ in $\Ss_s(\XX,1)$ is  a linear combination of (finite) sums of the form
$\sum_{a\in A_{+}} a^{q^d(q^{r_1}+\cdots+q^{r_s}-1)}$. But since $s\equiv 1 \mod q-1$, $q^d(q^{r_1}+\cdots+q^{r_s}-1)\equiv 0 \mod q-1$, and since $s \ge 2$, $q^d(q^{r_1}+\cdots+q^{r_s}-1)\neq 0$. Thus, by \cite[Example 8.13.9]{Goss_BasicStructures},  all the sums $\sum_{a\in A_{+}} a^{q^d(q^{r_1}+\cdots+q^{r_s}-1)}$ vanish, that is $\Ss_s(\XX, 1) = 0$.

Conversely,  the coefficient of $X_1\cdots X_s$ in $\Ss_s(\XX, 1)$ is $\sum_{a\in A_+} a^{s-1}$ which is congruent to $1$ modulo $\theta^q-\theta$ if $s=1$ or $s\not\equiv 1 \mod q-1$, so $\Ss_s(\XX, 1)$ does not vanish.
\end{proof}

\begin{exe}\label{exSs}
We already know that $\Ss_s(\XX,Z) = X_1\cdots X_s Z$ if $1\le s\le q-1$. Using Proposition \ref{propSs} and Lemma \ref{nulSs}, we easily see that 
\[\Ss_q(\XX,Z) = X_1\cdots X_q Z - X_1\cdots X_q Z^q.\]
For $q\ge 3$, a computation leads to
\[\Ss_{q+1}(\XX,Z) =  X_1\cdots X_{q+1} Z - X_1\cdots X_{q+1} (X_1^{q-1}+\cdots + X_{q+1}^{q-1})Z^q.\]
\end{exe}

\begin{lem}\label{Simon} 
Let $s\ge 1$,
\begin{enumerate}
\item for all integer $k\ge\frac s{q-1}$, the sum $\sum_{a\in A_{+,k}} a(t_1)\cdots a(t_{s-1})$ vanishes, so that $L(0,s-1) = \sum_{k \ge 0}\sum_{a\in A_{+,k}} a(t_1)\cdots a(t_{s-1})\in \F_q[\ttt]$,
\item $\Ss_s(\XX, 1) \equiv \left(L(0,s-1).(X_1\cdots X_{s-1})\right)X_s \mod X_s^q$.
\end{enumerate}
\end{lem}
\begin{proof}
For all $k\ge 0$,
$L_{k}(X_1\cdots X_s) = \sum_{a\in A_{+,k}} \frac{a*(X_1\cdots X_s)}a$ can be viewed as a polynomial in $X_s$, with no constant term, and since $C_a(X_s)\equiv a X_s \mod X_s^q$, we have :
\begin{eqnarray*}
Z_k(X_1\cdots X_s) &\equiv& L_{k}(X_1\cdots X_s) \equiv \sum_{a\in A_{+,k}} \frac{a*(X_1\cdots X_{s})}a    \mod X_s^q\\
&\equiv & \sum_{a\in A_{+,k}} a*(X_1\cdots X_{s-1}) \frac {aX_s}a \equiv X_s\sum_{a\in A_{+,k}} a*(X_1\cdots X_{s-1}) \mod X_s^q.
\end{eqnarray*}
But Proposition \ref{propSs} tells that $Z_k(X_1\cdots X_s) =0$ if $k\ge \frac s {q-1}$.
Remark now  that 
\[\sum_{a\in A_{+,k}} a*(X_1\cdots X_{s-1}) = \sum_{a\in A_{+,k}} a(t_1)\cdots a(t_{s-1}) . (X_1\cdots X_{s-1}).\]
We deduce then the first point from Lemma \ref{isomElin} and the evaluation at $Z=1$ :
\[L(0,s-1).(X_1\cdots X_{s-1})X_s\equiv\sum_{k\ge 0} Z_k(X_1\cdots X_s)\equiv \Ss_s(X,1) \mod X_s^q\]
gives the second point.
\end{proof}

Note that the first point of the above lemma is also a consequence of \cite[Lemma 8.8.1]{Goss_BasicStructures}  (see also \cite[Lemma 30]{AnglesPellarin_UniversalGT}).

\begin{lem}\label{Ssrad}
Let $s\ge 1$, if there exists $b,c\in A$ and $r\in X_s\Cinf[\XX]$ such that
\[ C_b(r) = C_{c} (\Ss_s(\XX, 1)) \]
then $b$ divides $c$ in $A$ and $r=C_{\frac cb}(\Ss_s(\XX, 1))$.
\end{lem}
\begin{proof}
We first prove that $r$ has coefficients in $A$. We will use the fact that $\exp_C$ and $\log_C$ define reciprocal bijections of $X_sK[X_1, \dots, X_{s-1}][[X_s]]$ satisfying for all $F\in X_sK[X_1, \dots, X_{s-1}][[X_s]]$ and $a\in A$, $\log_C (C_a(F)) = a\log_C(F)$ and $\exp_C (aF) = C_a(\exp_C(F))$. Thus $C_b(r) =\exp_C(b\log_C(r))$ and $C_{c} (\Ss_s(\XX, 1)) = \exp_C(c\log_C(\Ss_s(\XX, 1)))$. We deduce that $r= \exp_C(\frac cb \Ss_s(\XX, 1)) \in X_s K[\XX]$. But $C_b(X)$ is monic up to a unit in $\F_q^*$, and $A[\XX]$ is integrally closed. Thus the fact that $C_b(r) \in A[\XX]$ implies that $r\in A[\XX]$.

Write now $r \equiv X_s r_1\mod X_s^2$ with $r_1\in A[X_1, \dots, X_{s-1}]$. Then 
$C_b(r) \equiv X_s br_1 \mod X_s^2$ and by Lemma \ref{Simon}, 
\[C_{c} (\Ss_s(\XX, 1)) \equiv cX_s \left(L(0,s-1).(X_1\cdots X_{s-1})\right) \mod X_s^2,\] 
thus $r_1 = \frac cb L(0,s-1).(X_1\cdots X_{s-1})$. Since $r_1\in A[X_1, \dots, X_{s-1}]$, Lemma \ref{isomAlin} assures that $\frac cb L(0,s-1) \in A[t_1, \dots, t_{s-1}]$. But $L(0, s-1)\in \F_q[t_1, \dots, t_{s-1}]$. We obtain that $b$ divides $c$ in $A$ and that $r = \exp_C(\frac cb \Ss_s(\XX, 1)) = C_{\frac cb}(\Ss_s(\XX, 1))$.
\end{proof}

Set $\mathfrak R = \cup_{s\ge 1}\Cinf[X_1, \dots, X_s]$ and let $\mathfrak F$ be the sub-$A$-module of $C(\mathfrak R)$ generated by the polynomials $\Ss_s(X_1, \dots, X_s, 1), s\ge 1$. Set
\[\sqrt {\mathfrak F} = \left\{ r\in \mathfrak R, \exists a \in A^*, C_a(r) \in \mathfrak F\right\}.\]
\begin{thm}
\[\sqrt {\mathfrak F} = \mathfrak F + C(\Lambda_C).\]
\end{thm}
\begin{proof}
The inclusion $\mathfrak F + C(\Lambda_C)\subset \sqrt {\mathfrak F}$ is clear.

Let $r\in \sqrt {\mathfrak F}$. Then there exists $n\ge 1$ such that $r\in \Cinf[X_1, \dots, X_n]$ and there exist $a\in A^*$, $a_1, \dots, a_n \in A$ such that 
\begin{equation}\label{eqthmrad}
C_a(r) = \sum_{m=1}^n C_{a_m} (\Ss_m(X_1, \dots, X_m,1)).
\end{equation}
We now prove by induction on $n\ge 1$ that $r\in \mathfrak F + C(\Lambda_C)$.

In the case $n=1$, Equation \eqref{eqthmrad} reduces to 
\[ C_a(r)  = C_c(\Ss_1(X_1,1)) \]
with $c\in A$ and $r\in \Cinf(X_1)$. The constant term of $r$ is then in $C(\Lambda_C)$ and we can therefore assume $r\in X_1\Cinf[X_1]$. The result in this case is then just the one of Lemma \ref{Ssrad}.

We suppose now $n>1$ and that the result is proved for all $k\le n-1$. We can assume that $a_n\neq 0$ and $\Ss_n(X_1, \dots, X_n,1)\neq 0$, that is : $n \not \equiv 1 \mod q-1$. Write $r = \sum_{i=0}^d r_i(X_1, \dots, X_{n-1} )X_n^i$, with $d>0$. Then Equation  \eqref{eqthmrad} evaluated at $X_n=0$ yields :
\[C_a(r_0(X_1, \dots, X_{n-1} ))  =  \sum_{m=1}^{n-1} C_{a_m} \Ss_m(X_1, \dots, X_m,1)\]
and the induction hypothesis assures that $r_0(X_1, \dots, X_{n-1} ) \in \mathfrak F + C(\Lambda_C)$. Thus we can assume $r_0=0$ and, for some $c\in A$,
\[ C_a(r) = C_{c} (\Ss_n(X_1, \dots, X_n,1)) \]
Again, we are reduced to the result proved in Lemma \ref{Ssrad}.
\end{proof}

\subsection{Converses of the $\log$-algebraicity Theorem}
The $\log$-algebraicity theorem asserts that if $F\in A[\XX]$, then $\mathfrak L(F,Z) \in A[\XX, Z]$. We will prove in this section conversely that, if $F\in \Cinf[\XX]$ and $\mathfrak L(F,Z)$ belongs to $\Cinf[\XX,Z]$ or to $\overline{A} [\XX][[Z]]\otimes_A K$, then necessarily, $F\in A[\XX]$.

Let $\overline{A}$ be the integral closure of $A$ in $\overline{K}$. If $P$ is a prime of $A,$ $\overline{A}_{(P)}$ denotes the ring of elements of $\overline{K}$ that are $P$-integral.
\begin{lem}\label{convlem1}
Let $x\in A$ such that for infinitely many primes $P$
\[x^{q^d}\equiv x\pmod{P^p},\]
where $d$ is the degree of $P.$ Then $x \in A^p.$
\end{lem}
\begin{proof} Let $F\in A\setminus A^p,$ then $F' \not = 0,$ where $F'$ denotes the derivative of $F$ with respect to the variable $\theta.$ Then $F^{q^{d}}-F \equiv (\theta^{q^d}-\theta)F' \mod P^2$, so that for all primes $P$ not dividing $F'$, $v_P(F^{q^{d}}-F) =1.$
\end{proof}
\begin{lem}\label{convlem2}
\begin{enumerate}
\item  Let $\alpha \in \overline{A}$ such that for all but finitely many primes $P$ of $A$
\[\alpha ^{q^d} \equiv \alpha \pmod{P\overline{A}},\]
where $d$ is the degree of $P.$ Then $\alpha \in A.$
\item Let $\alpha \in \overline{K}$ such that for all but finitely many primes $P$ of $A$
\[\alpha ^{q^d} \equiv \alpha \pmod{P\overline{A}_{(P)}},\]
where $d$ is the degree of $P.$ Then $\alpha \in K.$
\end{enumerate}
\end{lem}
\begin{proof}
\begin{enumerate}
\item First we assume that $\alpha $ is separable over $K.$ Set $F=K(\alpha)$ and let $O_F$ be the integral closure of $A$ in $F.$  For a prime $P$ not dividing the discriminant of $A[\alpha ],$ we have :
\[O_F\otimes_AA_P = A[\alpha]\otimes_A A_P,\]
where $A_P$ is the $P$-adic completion of $A.$ Therefore, for all but finitely many primes $P,$ we have :
\[\forall x\in O_F, \, x^{q^d}\equiv x \pmod{PO_F}.\]
This implies that all but finitely many primes $P$ of $A$ are totally split in $F.$ By the \v Cebotarev density theorem (see for example \cite[Chapter VII, Section 13]{Neukirch_ANT}), this implies that $F=K$ and thus $\alpha \in A.$

In general there exists a minimal integer $m\geq 0$ such that $\alpha^{p^m}$ is separable over $K$. If $m\ge 1$, then $x=\alpha^{p^m}\in A$ and for all but finitely many primes $P$ of $A:$
\[x^{q^d}\equiv x\pmod{P^{p^m}A}.\]
Therefore $\alpha^{p^{m-1}} \in A$ by lemma \ref{convlem1}. We deduce that $\alpha\in A$.
\item Let $b\in A\setminus \{0\}$ such that $x=b\alpha \in \overline{A}.$  Then  by the first assertion of the Lemma, $x\in A.$ Therefore $\alpha \in K.$ 
\end{enumerate}
 \end{proof}
\begin{prop}
\label{ConvAlg1} For all $s\ge 1$, if $\XX =(X_1, \dots, X_s)$, then
\[\left\{F \in \Cinf[\XX] ; \mathfrak L(F,Z) \in \overline{A} [\XX][[Z]]\otimes_A K \right\} = A[\XX].\]
\end{prop}
\begin{proof}  Let $F\in \Cinf[\XX]$ such that $\frak L(F,Z) \in \overline{A} [\XX][[Z]]\otimes_AK,$ \emph{i.e.}  there exists $b\in A\setminus\{0\}$ such that $b\frak L(F,Z) \in  \overline{A} [\XX][[Z]].$ Since $\frak L(F,Z) \equiv FZ \pmod{Z^q},$ we get $F\in \overline{K} [\XX].$ Let $P$ be a prime of $A$ of degree $d$ not dividing $b$. Then by the proof of Lemma \ref{logalglem3},
\[\mathfrak L(F,Z)\in \overline{A}_{(P)}[\XX][[Z]] \textrm{ and }(C_P-P*)\left(\mathfrak L(F,Z)\right)\in P\overline{A}_{(P)}[\XX][[Z]]\]
and since $C_P \equiv \tau^d \mod PA[\tau]$, the coefficient of $Z^{q^d}$ in $(C_P-P*)\left(\mathfrak L(F,Z)\right)$ is congruent to 
$ F^{q^d}-P*F \mod  P\overline{A}_{(P)}[\XX][[Z]]$.
Therefore
\[F(X_1^{q^d}, \dots , X_s^{q^d}) \equiv F^{q^d} \mod{P\overline{A}_{(P)}[\XX]}.\]
Thus, by Lemma \ref{convlem2}, we get $F\in K[\XX].$
Select now $c\in A\setminus \{0\}$ such that $cF\in A[\XX].$ Then by Theorem \ref {logalgthm}:
\[C_c(\frak L(F,Z))\in A[X_1, \dots , X_s, Z].\]
Therefore $\frak L(F,Z)\in A[\XX][[Z]]\otimes_A K$ is integral over $A[\XX][[Z]].$ But $A[\XX][[Z]]$ is integrally closed (see\cite[Chapitre 5, Proposition 14]{Bourbaki_AC56}) thus $\frak L(F,Z) \in A[\XX][[Z]]$ and this implies that $F\in A[\XX]$ since $\frak L(F,Z)\equiv FZ \mod Z^q$. We then have the direct inclusion, the equality follows by Theorem \ref{logalgthm}.
\end{proof}
Remark that if we only suppose that $\mathfrak L(F,Z)\in \overline{K}[\XX][[Z]]$, then the result no longer holds, for instance $F=\frac X\theta \in K[X]\backslash A[X]$ and $\mathfrak L(F,Z) \in K[X][[Z]]$.  Note that the above Proposition implies that $\frak L^{-1} (\overline{K} [\XX,Z])= A[\XX].$ In fact we have :
\begin{prop}
\label{ConvAlg2}
\[\left\{F \in \Cinf[\XX] ; \mathfrak L(F,Z) \in \Cinf [\XX,Z] \right\} = A[\XX].\]
\end{prop}
\begin{proof}
Recall that if $\ii =(i_1, \dots, i_s)\in \N^s$, then $\XX^\ii = X_1^{i_1}\cdots X_s^{i_s}$. If $F \in \Cinf[\XX]$, write $F= \sum_{\ii}\alpha_\ii \XX^\ii$ and define $\mdeg(F) \in \N^s_+\cup\{\pm\infty\}$ to be the maximum for the lexicographic ordering of the exponents $\ii$ such that $\alpha_\ii \neq 0$. Now, if $F =\sum_{k\ge 1}F_kZ^k \in Z\Cinf[\XX][[Z]]$ where for all $k$, $F_k\in \Cinf[\XX]$, then we define the relative degree of $F$, $\rdeg(F) \in \R^s_+\cup\{\pm\infty\}$, to be
\[\rdeg(F) = \left\{ 
\begin{array}{cc}
-\infty & \textrm{ if }F=0\\
\displaystyle{\sup_{k\ge 1}\left(\frac{\mdeg(F_k)}{k}\right) \in \R_+^s\cup\{+\infty\}} & \textrm{otherwise.}
\end{array}\right.\]
where the supremum is still relative to the lexicographic ordering on $\R_+^s$ and is well defined in $\R_+^s\cup\{+\infty\}$. Remark the following properties of $\rdeg$ : if $F,G \in Z\Cinf[\XX][[Z]]$ and $\psi$ is an $\F_q$-linear power series in $\Cinf[[T]]$ then
\begin{itemize}
\item $\rdeg(F+G) \le \max(\rdeg(F) , \rdeg(G))$ with equality if $\rdeg(F) \neq \rdeg(G)$,
\item $\rdeg(F^q)=\rdeg(F)$,
\item $\rdeg(\psi(F)) \le \rdeg(F)$,
\item if $\psi\neq 0$, for $k\ge 1$ and $\ii\in \N^s$, $\rdeg(\psi(\XX^\ii Z^k)) = \frac \ii k $,
\item if $F =\sum_{k\ge 1} F_kZ^k$ is such that there exists infinitely many indices $k_j$ such that $\mdeg(F_{k_j})= k_j\rdeg(F)$ (in particular $F\notin A[\XX,Z]$) and $\rdeg(F) > \rdeg(G)$ then $F+G \notin  A[\XX,Z]$.
\end{itemize}
For the last property, if we write $G =\sum_{k\ge 1} G_kZ^k$, then $F+G =\sum_{k\ge 1} (F_k+G_k)Z^k$ with for all $j$, $\mdeg(F_{k_j}+G_{k_j}) = k_j\rdeg(F)$ so that $F_{k_j}+G_{k_j}\neq 0$ and $F+G \notin  A[\XX,Z]$.

Let now $\ii\in \N^s$, for $k\ge 0$, we have 
\[L_k(\XX^\ii) = \frac{\XX^{q^k\ii}}{l_k}+G_{k,\ii}\]
where $G_{k,\ii}\in K[\XX]$ satisfies $\mdeg(G_{k,\ii}) <q^k\ii$. Thus
\begin{equation}\label{Ldegrel}
\mathfrak L(\XX^\ii,Z) = \XX^\ii Z+F_\ii
\end{equation}
where $F_\ii \in Z^q\Cinf[\XX,Z]$ has relative degree $\rdeg(F_\ii)<\ii$.

Fix $\alpha\in \Cinf$, then
\[C_\alpha(T) = \exp_C(\alpha\log_C(T)) \in \Cinf[[T]]\]
is a $\F_q$-linear power series, and $\C_\alpha(T)\in \Cinf[T]$ if and only if $\alpha\in A$ (see \cite[Chapter 3]{Goss_BasicStructures}).

Let now $F \in \Cinf[\XX]\backslash A[\XX]$, we want to prove that $\mathfrak L(F,Z) \notin \Cinf[\XX,Z]$. By Theorem \ref{logalgthm}, we can suppose 
$F = \sum_{\ii}\alpha_\ii \XX^\ii$ with for all $\ii$ such that $\alpha_\ii \neq 0$, $\alpha_\ii \notin A$.
Then Equation \eqref{Ldegrel} gives
\[\mathfrak L(F,Z) = \sum_{\ii}C_{\alpha_\ii} \left(\mathfrak L(\XX^\ii, Z)\right) = \sum_{\ii} C_{\alpha_\ii} \left(\XX^\ii Z\right) + C_{\alpha_\ii} \left(F_\ii\right).\]
If $\ii_0=\mdeg(F)$, then we deduce that
\[\mathfrak L(F,Z) =C_{\alpha_{\ii_0}} \left(\XX^{\ii_0}Z\right)+G\]
with $\rdeg(G)<\ii_0$. Since $C_{\alpha_{\ii_0}}(\XX^{\ii_0}Z) \notin \Cinf[\XX,Z]$ has infinitely many terms of relative degree $\ii_0$, we have $\mathfrak L(F,Z) \notin \Cinf[\XX,Z]$.
\end{proof}

\section{Multivariable $L$-functions}
\subsection{Frobenius actions}
Let $\Kinf\subseteq E\subseteq \Cinf$ be a field complete with respect to $|\cdot|_\infty$.
Observe that if $n\ge 0$ and $1\le i \le s$, 
\[H_{q^n}(X_i)^q = C_{\theta^n}(X_i)^q = C_{\theta^{n+1}}(X_i) -\theta C_{\theta^n}(X_i) = (t_i-\theta).H_{q^n}(X_i).\]
Thus we define the following action of $\tau$ on $\TT_s(E)$ :
\[\forall f=\sum_\ii f_\ii \ttt^\ii \in \TT_s(E), \ \ \tau(f) = (t_1-\theta)\cdots (t_s-\theta) \sum_\ii f_\ii^q \ttt^\ii\]
and we get for all $f\in \TT_s(E)$ the equality in $ \BB_s^{\lin}(E)$ :
\[ \tau(f.(X_1\cdots X_s)) = \tau(f) .( X_1\cdots X_s).\]
We define then on $\TT_s(E)$ the operator $\varphi$ which will be a Frobenius acting only on coefficients, namely :
\[\forall f=\sum_\ii f_\ii \ttt^\ii \in \TT_s(E), \ \ \varphi(f) = \sum_\ii f_\ii^q \ttt^\ii,\]
so that on $ \TT_s(E)$, we have $\tau = (t_1-\theta)\cdots (t_s-\theta) \varphi$.
Moreover, for $d\ge 1$, if we define,  $b_d(t) = (t-\theta) (t-\theta^q)\cdots(t-\theta^{q^{d-1}})$, then for all $f \in \TT_s(E)$,
\[ \tau^d(f)= b_d(t_1)\cdots b_d(t_s)\varphi^d(f).\]
We also set $b_0(t) = 1$ so that the above relation still holds for $d=0$. Remark that for all $f,g\in \TT_s(E)$, and $d\ge 0$,
\[\tau^d(fg) = \tau^d(f) \varphi^d(g).\]
Observe moreover that
\[\forall f\in \TT_s(E), \ \forall d\ge 0, \ \ \  \|\varphi^d(f)\|_\infty =  \|f\|_\infty^{q^d} \textrm{ and }\|\tau^d(f)\|_\infty = q^{s\frac{q^d-1}{q-1}} \|f\|_\infty^{q^d}.\]
We deduce that $\exp_C =\sum_{j\ge 0} \frac 1 {D_j} \tau^j$ is defined on $\TT_s(E)$ and that for all $f\in \TT_s(E)$, we have in $\BB_s^{\lin}(E)$ :
\begin{equation}\label{equaexp}
\exp_C(f.(X_1\cdots X_s)) = \exp_C(f).(X_1\cdots X_s).
\end{equation}

\bigskip

We extend now the action of $E[\ttt]$ on $E[\XX]$ to an action of $E[\ttt][[z]]$ on $E[\XX][[Z]]$ via :
\[\left(\sum_{k\ge 0} f_k(\ttt) z^k \right). \left(\sum_{n\ge 0} F_n(\XX)Z^n \right)= \sum_{k\ge 0} \sum_{n\ge 0} \left(f_k(\ttt).F_n(\XX)\right) Z^{nq^k}\]
and we let $\tau$ act on $Z$ via $\tau (Z) = Z^q$. Since $\tau(Z)=z.Z$, we define on $E[\ttt][[z]]$ the operator $\tau_z$ :
\[\forall f=\sum_{k\ge 0}\sum_\ii f_\ii \ttt^\ii z^k \in E[\ttt][[z]], \ \ \tau_z(f) = z(t_1-\theta)\cdots (t_s-\theta) \sum_{k\ge 0}\sum_\ii f_\ii^q \ttt^\ii z^k = 
\sum_{k\ge 0}z^{k+1}\tau\left(\sum_\ii f_\ii^q \ttt^\ii\right).\]
Thus if we extend $\varphi$ by
\[\forall f=\sum_{k\ge 0}\sum_\ii f_\ii \ttt^\ii z^k \in E[\ttt][[z]], \ \ \varphi(f) = \sum_{k\ge 0}\sum_\ii f_\ii^q \ttt^\ii z^k,\]
we get for all $f=\sum_{k\ge 0}f_kz^k\in  E[\ttt][[z]]$ and $d\ge 0$, $\tau_z^d(f) = z^db_d(t_1)\cdots b_d(t_s)\varphi^d(f)$.
By construction, if $f=E[\ttt][[z]]$, then $f.(X_1\cdots X_sZ) \in E[\XX][[Z]]$ and for all $d\ge 0$,
\[\tau^d(f.(X_1\cdots X_sZ) ) = \tau_z^d(f).(X_1\cdots X_sZ).\]
We have then an operator $\exp_C =\sum_{j\ge 0} \frac 1 {D_j} \tau^j$ on $ZE[\XX][[Z]]$ and an operator $\exp_z =\sum_{j\ge 0} \frac 1 {D_j} \tau_z^j$ on $E[\ttt][[z]]$ such that for all $f\in E[\ttt][[z]]$,
\begin{equation}\label{equaexpz}
\exp_C(f.(X_1\cdots X_sZ)) = \exp_z(f).(X_1\cdots X_sZ).
\end{equation}
A similar property holds for $\log_C =\sum_{j\ge 0} \frac 1 {l_j} \tau^j$ and $\log_z =\sum_{j\ge 0} \frac 1 {l_j} \tau_z^j$ :
\begin{equation}\label{equalogz}
\log_C(f.(X_1\cdots X_sZ)) = \log_z(f).(X_1\cdots X_sZ)
\end{equation}
where $\log_z$ and $\exp_z$ define reciprocal bijection of $E[\ttt][[z]]$.

We now state compatibility results for evaluations at $Z=1$ and $z=1$ :
\begin{lem}\label{lemevalz}
Let $F(\XX,Z) = \sum_{n\ge 0}F_n(\XX) Z^n\in E[\XX][[Z]]$ with $\forall n\ge 0, F_n(\XX) \in E[\XX]$ and $\lim_{n \to \infty}\|F_n \|=0$, let $f =\sum_{k \ge 0}f_kz^k \in E[\ttt][[z]]$ with $\forall k \ge 0, f_k\in E[\ttt]$ and $\lim_{k \to \infty}\|f_k \|_\infty=0$, then $F$ and $f.F$ converge in $\BB_s(E)$ at $Z=1$, $f$ converges in $\TT_s(E)$ at $z=1$, and we have the following equality in $\BB_s(E)$ :
\[\left(f.F(\XX,Z)\right)_{|Z=1} = f(\ttt,1).F(\XX,1).\]
\end{lem}
\begin{proof}
The convergence of $F$ at $Z=1$ and of $f$ at $z=1$ are obvious, the convergence of $f.F$ follows from the equality $\|f_k.F_n\| = \|f_k\|_\infty\|F_n\|$ from Lemma \ref{lemNormGauss}. Finally, both sides of the equality are equal to $\sum_{k\ge 0}\sum_{n\ge 0} f_k.F_n$.
\end{proof}
\begin{lem}\label{lemevaltau}
Let $\eta=\sum_{n\ge 0} \eta_n \tau_z^n \in E[[\tau_z]]$ and $\eta^{1} =\sum_{n\ge 0} \eta_n\tau^n \in E[[\tau]]$, let $f =\sum_{k \ge 0}f_kz^k \in E[\ttt][[z]]$ with $\lim_{k \to \infty}\|f_k \|_\infty=0$, write $M= \sup_{k \ge 0}\|f_k \|_\infty$ and suppose $\lim_{n \to \infty}|\eta_n| (q^{\frac{s}{q-1}}M)^{q^n}=0$ ; 
write finally $g(\ttt,z) = \eta(f(\ttt,z)) =  \sum_{n\ge 0} \eta_n \tau_z^n(f) \in E[\ttt][[z]]$ ;
then $f$ and $g$ converge in $\TT_s(E)$ at $z=1$ and we have the following equality in $\TT_s(E)$ :
\[\eta^{1}(f(\ttt,1)) = g(\ttt,1).\]
\end{lem}
\begin{proof} The convergence of $f$ is obvious.
Both sides of the above equality are easily seen to be equal to
\[ \sum_{n\ge 0, k\ge 0} \eta_n \tau^n(f_k(\ttt)) \]
which is the sum of a summable family in $\TT_s(E)$. This gives at once the convergence of both sides of the equality and the desired identity.
\end{proof}

Define now
\[E[\XX][[Z]]^{\lin} = \left\{F\in E[\XX][[Z]] ; F \text{ is linear with respect to each of the variables } X_1, \dots, X_s,Z \right\}.\]

\begin{lem}\label{leminjZ}
\begin{enumerate}
\item The map 
\[\left\{\begin{array}{ccc}
E[\ttt][[z]] & \rightarrow & E[\XX][[Z]] \\
f & \mapsto & f.(X_1\cdots X_sZ)
\end{array}\right.\]
is injective with image $E[\XX][[Z]]^{\lin}$,
\item $f\in E[\ttt][[z]]$ satisfies $f.(X_1\cdots X_sZ) \in A[\XX][Z]$ if, and only if, $f\in A[\ttt][z]$.
\end{enumerate}
\end{lem}
\begin{proof}
The first point is an immediate consequence of Lemma \ref{isomElin} and the second one a consequence of Lemma \ref{isomAlin}.
\end{proof}

\subsection{Anderson-Stark units}

We define for all integers $N \in \Z$, $s\ge 1$ :
\[L(N,s,z) = \sum_{d\ge 0} z^d \sum_{a\in A_{+,d}} \frac{a(t_1)\dots a(t_s)}{a^N} \in K[\ttt][[z]]\]
and 
\[L(N,s) = \sum_{d\ge 0} \sum_{a\in A_{+,d}} \frac{a(t_1)\dots a(t_s)}{a^N} \in \TT_s(\Kinf)\]
where $L(N,s,z) \in A[\ttt,z]$ if $N\le 0$ because of Lemma \ref{Simon}.
We also define the operator $\Gamma$ :
\[ \forall F\in \BB_s(\Cinf), \ \  \Gamma(F) = L(1,s).F \in \BB_s(\Cinf).\]
We will refer to $\Gamma$ as to \emph{Anderson's operator}.
Remark that $L(1,s)$ has norm $1$, so that $\Gamma$ is an isometry of $\BB_s(\Cinf)$, in particular, $\left\|\Gamma(X_1\cdots X_s)\right\| = q^{\frac s{q-1}}$.
We define now
\[ \sigma_s(\ttt,z) = \exp_z(L(1,s,z)).\]
We then have :
\begin{prop}\label{propLsigma}
\begin{enumerate}
\item $ L(1,s,z) = \log_z(\sigma_s(\ttt,z))$,
\item $L(1,s,z).X_1\cdots X_sZ = \mathfrak L(X_1\cdots X_s, Z) = \log_C(\Ss_s(\XX,Z))$,
\item $\sigma_s(\ttt,z). X_1\cdots X_s Z =  \Ss_s(\XX,Z)$ and $\sigma_s(\ttt,z) \in A[\ttt,z]$.
\end{enumerate}
\end{prop}
\begin{proof}
The first point and equality $L(1,s,z).X_1\cdots X_sZ = \mathfrak L(X_1\cdots X_s, Z)$ are clear. The equality $L(1,s,z).X_1\cdots X_sZ= \log_C(\Ss_s(\XX,Z))$ comes from Equation \eqref{equalogz}.
Equation $\eqref{equaexpz}$ shows that $\sigma_s(\ttt,z). X_1\cdots X_s Z =  \Ss_s(\XX,Z)$ and the fact that $\sigma_s(\ttt,z) \in A[\ttt,z]$ is a consequence of Lemma \ref{leminjZ} and Theorem \ref{logalgthm}.
\end{proof}
We call the special polynomial $\sigma_s(\ttt,z)$ the \emph{Anderson-Stark unit of level $s$}.

The evaluation at $Z=1$ leads to :
\begin{prop}
\begin{enumerate}
\item $\exp_C(\Gamma(X_1\cdots X_s)) = \Ss_s(\XX,1)$,
\item if $s<q$, then $\Gamma(X_1\cdots X_s) = \log_C(X_1\cdots X_s)$.
\end{enumerate}
\end{prop}
\begin{proof}
Lemma \ref{lemevaltau} shows that $\sigma_s(\ttt,1) = \exp_C(L(1,s))$, and Equation \ref{equaexp} yields to the first point.
For the second point, we remark that if $s<q$, then $\Ss_s(\XX,Z)=X_1\cdots X_sZ $ so that 
\[ \mathfrak L(X_1\cdots X_s, Z) = \log_C(\Ss_s(\XX,Z)) = \sum_{n\ge 0} \frac {(X_1 \cdots X_s Z)^{q^n}}{l_n}\]
but $\|X_1\cdots X_s\|=q^{\frac s {q-1}}<q^{\frac q {q-1}}$ so that $\Gamma(X_1\cdots X_s) = \sum_{n\ge 0} \frac {(X_1 \cdots X_s)^{q^n}}{l_n} = \log_C(X_1\cdots X_s)$ converges in $\BB_s(\Kinf)$.
\end{proof}
We can recover properties of $\sigma_s$ from the ones of $\Ss_s$.
\begin{prop}
\begin{enumerate}
\item $\deg_z(\sigma_s(\ttt,z)) \le \frac {s-1}{q-1}$,
\item $z-1$ divides $\sigma_s(\ttt,z)$ if, and only if, $s\equiv 1 \mod q-1$ and $s>1$.
\end{enumerate}
\end{prop}
\begin{proof}
The first point comes from Proposition \ref{propSs} and the second one from Lemma \ref{nulSs}.
\end{proof}
Note that in practice, formulas for $\sigma_s$ are more manageable and easier to compute than the formulas for $\Ss_s$. Compare the following example with Example \ref{exSs} :
\begin{exe}
If $1\le s \le q-1$, then $\sigma_s=1$. The next two polynomials are : $\sigma_q = 1-z$ and $\sigma_{q+1} = 1-(t_1-\theta)\cdots(t_{q+1}-\theta)z$.
\end{exe}

In the spirit of Lemma \ref{Simon}, we can recover the values $L(N,s,z)$ for $N\le 0$ from the polynomials $\Ss_s(\XX,Z)$ :
\begin{thm}
For all $N\ge 0$ and $s\ge 1$,
\[L(-N,s,z).(X_1\cdots X_s) =  \frac d{dX_{s+1}} \cdots \frac d{dX_{s+N+1}} \Ss_{s+N+1}(X_1, \cdots, X_{s+N+1},Z). \]
\end{thm}
\begin{proof}
Since $\frac d {dX}(a*X) = a$, we have :
\[ \frac d{dX_{s+1}} \cdots \frac d{dX_{s+N+1}} Z_k(X_1\cdots X_{s+N+1}) = \sum_{a\in A_{+,k}} a*(X_1\cdots X_{s})a^{N}\]
which gives the result.
\end{proof}

\section{Special $L$-values}

The purpose of this section is to express the series $L(N,s,z)$ as sums of polylogarithms. The idea here is to use the fact that if we evaluate $t_{n+1}, \dots, t_s$ at $\theta$ in $\varphi^r(L(1,s,z)) = L(q^r,s,z)$, we just obtain $L(q^r+n-s,n,z)$.

If $P$ is a polynomial in a variable among $t, t_1, \dots,$ or $t_s$, we will write $P^{\varphi}$ for $\varphi(P)$.

\begin{lem}
For all integers $k\ge 0$ and $r\ge 1$, we have :
\begin{enumerate}
\item $ \varphi^r(b_k(t)) = \frac {b_{k+r}}{b_r} = \frac{b^\varphi_{k+r-1}}{b_{r-1}^\varphi}$,
\item $b_{k+r}(t) = \varphi^k(b_r(t)) b_k(t)$,
\item $b_k^\varphi(\theta) = l_k$.
\end{enumerate}
\end{lem}
\begin{proof}
The verification of these identities is left to the reader.
\end{proof}

We start from the first point of Proposition \ref{propLsigma}, and we write $\sigma_s(\ttt,z) = \sum_{i=0}^m \sigma_{s,i}(\ttt)z^i$ :
\begin{eqnarray*}
L(1,s,z) &=& \log_z(\sigma_s(\ttt,z)) \\
&=& \sum_{k\ge 0} \frac 1 {l_k} \tau_z^k(\sigma_s(\ttt,z)) = \sum_{k\ge 0} \sum_{i=0}^m\frac{z^{k+i}}{l_k} \tau^k( \sigma_{s,i}(\ttt))\\
&=& \sum_{k\ge 0} \sum_{i=0}^m\frac{z^{k+i}}{l_k} b_k(t_1)\cdots b_k(t_s)\varphi^k( \sigma_{s,i}(\ttt)).
\end{eqnarray*}
If we apply $\varphi^r$ on both sides, we get :
\begin{eqnarray*}
L(q^r,s,z) &=& \sum_{k\ge 0} \sum_{i=0}^m\frac{z^{k+i}}{l_k^{q^r}} \varphi^r(b_k(t_1)\cdots b_k(t_s))\varphi^{k+r}( \sigma_{s,i}(\ttt))\\
&=& \sum_{i=0}^m\sum_{k\ge 0}  \frac{z^{k+i}}{l_k^{q^r}} \frac {(b_{k+r}(t_1)\cdots b_{k+r}(t_n))b_{k+r-1}^\varphi(t_{n+1})\cdots b_{k+r-1}^\varphi(t_{s})}{b_r(t_1)\cdots b_r(t_n).b_{r-1}^\varphi(t_{n+1})\cdots b_{r-1}^\varphi(t_{s})}  \varphi^{k+r}( \sigma_{s,i}(\ttt)).
\end{eqnarray*}
Write $\sigma_{s,i}(\ttt) = \sum_{i_{n+1}, \dots, i_{s}} f_{i_{n+1}, \dots, i_s} t_{n+1}^{i_{n+1}}\cdots t_s^{i_s}$ with $f_{i_{n+1},\ldots,i_s}\in A[t_1,\ldots,t_n]$ and for $j\ge 0$,
$g_{i,j} = \sum_{i_{n+1}+\cdots+ i_{s}=j} f_{i_{n+1}, \dots, i_s}$ so that $\sigma_{s,i}(\ttt)$ evaluated at $t_{n+1} = \cdots = t_s = \theta$ is the polynomial $\sum_{j\ge 0}\theta^j g_{i,j}$.
We now evaluate $L(q^r,s,z)$ at $t_{n+1} = \cdots = t_s = \theta$ and we write $N=q^r-s+n$ :
\begin{eqnarray*}
L(N,n,z) &=& \sum_{i=0}^m\sum_{k\ge 0}  \frac{z^{k+i}}{l_k^{N}l_{r-1}^{s-n}} \frac {b_{k+r}(t_1)\cdots b_{k+r}(t_n)}{b_r(t_1)\cdots b_r(t_n)} \sum_{j\ge 0}\theta^j \varphi^{k+r}(g_{i,j})\\
&=&  \sum_{j\ge 0}\theta^j \sum_{i=0}^m \sum_{k\ge 0}  \frac{z^{k+i}}{l_k^{N}l_{r-1}^{s-n}} \frac {b_{k+r}(t_1)\cdots b_{k+r}(t_n)}{b_r(t_1)\cdots b_r(t_n)} \varphi^{k+r}(g_{i,j})\\
&=&  \sum_{j\ge 0}\theta^j \sum_{i=0}^m \sum_{k\ge 0}  \frac{z^{k+i}}{l_k^{N}l_{r-1}^{s-n}} \frac {b_{k}(t_1)\cdots b_{k}(t_n)}{b_r(t_1)\cdots b_r(t_n)} \varphi^{k}(b_r(t_1)\cdots b_r(t_n)\varphi^r(g_{i,j}))
\end{eqnarray*}
Write now $\log_{N,z} = \sum_{k\ge 0} z^k\frac{b_k(t_1)\cdots b_k(t_n)}{l_k^N} \varphi^k = \sum_{k\ge 0}\frac 1{l_k^N}\tau_z^k$, then
\begin{eqnarray*}
L(N,n,z) &=& \frac 1{l_{r-1}^{s-n}b_r(t_1)\cdots b_r(t_n)} \sum_{j\ge 0}\theta^j  \log_{N,z} \left(\sum_{i=0}^m z^ib_r(t_1)\cdots b_r(t_n)\varphi^r(g_{i,j})\right). \end{eqnarray*}
We have proved :

\begin{thm}\label{polylogthm}
For all integers $N\in \Z$, $n\ge 1$ and $r\ge 1$ such that $q^r\ge N$,
there exist integers $m,d\ge0$, and for $0\le i \le m$, $0\le j \le d$, polynomials $g_{i,j}\in A[t_1, \dots, t_n]$ such that if for $0\le j \le d$, 
$h_j = \sum_{i=0}^m z^i b_r(t_1)\cdots b_r(t_n)\varphi^r(g_{i,j})  =  \sum_{i=0}^m z^i \tau^r(g_{i,j})$, then
\[L(N,n,z) =\frac 1{l_{r-1}^{q^r-N}b_r(t_1)\cdots b_r(t_n)} \sum_{j=0}^d\theta^j \log_{N,z}(h_j).\]
\end{thm}

Denote now for $N\in \Z$, $\log_N = \sum_{k\ge 0}\frac 1{l_k^N}\tau^k$.

\begin{cor}
For all integers $N\in \Z$, $n\ge 1$ and $r\ge 1$ such that $q^r\ge N$, there exist integers $m,d\ge0$, and for $0\le i \le m$, $0\le j \le d$, polynomials $G_{i,j}\in A[X_1, \dots, X_n]^{\lin}$ such that if for $0\le j \le d$, 
$H_j = \sum_{i=0}^m Z^{q^i} \tau^r(G_{i,j})$, then
\[L(N,n,z) .\left( X_1^{q^r}\cdots X_n^{q^r}\right)=\frac 1{l_{r-1}^{q^r-N}} \sum_{j= 0}^d\theta^j \log_{N}(H_j).\]
\end{cor}

\end{document}